\documentclass{amsart}
\usepackage{amssymb,mathtools}
\usepackage[british]{babel}
\usepackage{enumitem}
\usepackage[latin1]{inputenc}
\usepackage{bussproofs}
\usepackage{graphicx}
\usepackage{hyperref}

\newtheorem{theorem}{Theorem}
\numberwithin{theorem}{section}
\newtheorem{corollary}[theorem]{Corollary}
\newtheorem{lemma}[theorem]{Lemma}
\newtheorem{proposition}[theorem]{Proposition}
\theoremstyle{definition}
\newtheorem{definition}[theorem]{Definition}

\newenvironment{bprooftree}
  {\leavevmode\hbox\bgroup}
  {\DisplayProof\egroup}

  {\gdef\scalefactor{#1}\begin{center}\proofSkipAmount \leavevmode}%
  {\scalebox{\scalefactor}{\DisplayProof}\proofSkipAmount \end{center} }

\numberwithin{equation}{section}
\allowdisplaybreaks

\newcommand{\isigma}{\mathbf{I\Sigma}}
\newcommand{\idelta}{\mathbf{I\Delta}}
\newcommand{\pa}{\mathbf{PA}}
\newcommand{\pr}{\text{Pr}}

\newcommand{\feps}{F_{\varepsilon_0}}
\newcommand{\true}{\text{True}}
\newcommand{\false}{\text{False}}

\newcommand{\rfn}{\text{RFN}}

\newcommand{\add}{\text{Add}}
\newcommand{\mult}{\text{Mult}}

\newcommand{\dterm}{\text{td}}
\newcommand{\height}{\text{ht}}

\newcommand{\ul}[1]{\boxed{#1}}

\newcommand{\ax}{\text{Ax}}
\newcommand{\ind}{\text{Ind}}
\newcommand{\sind}{\textbf{ind}}
\newcommand{\cut}{\text{Cut}}
\newcommand{\scut}{\textbf{cut}}
\newcommand{\zet}{\mathbf{Z}}

\newcommand{\acc}{\text{Acc}}

\newcommand{\pend}{\text{End}}
\newcommand{\prule}{\text{Rule}}
\newcommand{\pord}{\text{Ord}}
\newcommand{\pred}{\text{Pred}}
\newcommand{\fund}{\text{Fund}}
\newcommand{\rep}{\text{Rep}}
\newcommand{\dcut}{d_\text{cut}}
\newcommand{\dacc}{d_\text{acc}}

\newcommand{\id}{\text{Id}}
\newcommand{\step}{\text{Step}}
\newcommand{\lc}{\text{LC}}
\newcommand{\stepto}{\text{top}}
\newcommand{\stepbo}{\text{bottom}}
\newcommand{\stepba}{\text{base}}
\newcommand{\stepproof}{\text{StepProof}}
\newcommand{\stepbound}{\text{StepBound}}
\newcommand{\seqproof}{\text{SeqProof}}
\newcommand{\seqbound}{\text{SeqBound}}

\makeatletter
\newcommand{\dotminus}{\mathbin{\text{\@dotminus}}}
\newcommand{\@dotminus}{%
  \ooalign{\hidewidth\raise1ex\hbox{.}\hidewidth\cr$\m@th-$\cr}%
}
\makeatother

\title[Characterizing $\Sigma_1$-Reflection over the Fragments of $\pa$]{A Uniform Characterization of $\Sigma_1$-Reflection\\ over the Fragments of Peano Arithmetic}
\author{Anton Freund}

\begin{document}

\begin{abstract}
We show that the theory $\isigma_1$ of $\Sigma_1$-induction proves the following statement: For all $n\geq 2$, the uniform $\Sigma_1$-reflection principle over the theory $\isigma_n$ is equivalent to the totality of the function $F_{\omega_n}$ at stage $\omega_n$ of the fast-growing hierarchy. The method applied is a formalization of infinite proof theory. The literature contains several proofs which place the quantification over $n$ in the meta-theory (and also prove the separate cases $n=0,1$). In contrast, the author knows of no explicit argument that would allow us to internalize the quantification while keeping the meta-theory as low as $\isigma_1$. It is well possible that this has been considered before. Our aim is merely to provide a detailed exposition of this important result.
\end{abstract}

\maketitle

\tableofcontents

Recall the uniform $\Sigma_1$-reflection principle over the theory $\isigma_n$. It is given by the $\Pi_2$-formula
\begin{equation*}
\rfn_{\isigma_n}(\Sigma_1):\equiv\forall_\varphi(\text{``$\varphi$ a closed $\Sigma_1$-formula"}\land\pr_{\isigma_n}(\varphi)\rightarrow\true_{\Sigma_1}(\varphi)),
\end{equation*}
using a standard truth definition $\true_{\Sigma_1}(\cdot)$ for $\Sigma_1$-formulas. Recall also the fast-growing hierarchy of functions $F_\alpha$ indexed by ordinals. The precise definition of these functions varies slightly throughout the literature; we adopt the version of \cite{buchholz-wainer-87}, or also \cite{sommer95}. In \cite{sommer95} one can find a $\Delta_0$-formula $F_\alpha^i(x)=y$ in the variables $\alpha$, $i$, $x$ and $y$ which defines the functions $F_\alpha$ (and their iterations) for $\alpha<\varepsilon_0$. Many relations between these functions (but, of course, not the fact that they are total) are then provable in $\isigma_1$.\\
The literature contains several proofs of the equivalence
\begin{equation*}
 \isigma_1\vdash\rfn_{\isigma_n}(\Sigma_1)\leftrightarrow F_{\omega_n}\!\downarrow
\end{equation*}
for fixed $n$. Famously, Paris in \cite{paris80} gave a model-theoretic argument for this fact. Sommer in \cite{sommer95} shows the same result over an even weaker base theory. For a proof theoretic approach (via iterated reflection principles) we refer to \cite[Theorem 1, Proposition 7.3, Remark 7.4]{beklemishev03}. According to H\'ajek and Pudl\'ak's comment after \cite[Theorem II.3.36]{hajek91} it is expected but not quite trivial that the quantification can be internalized. It is well possible that this has now been checked for some of the cited proofs. The aim of the present paper is to present a rigorous argument (by a different method) in detail. In fact, we only prove the direction
\begin{equation*}
\isigma_1\vdash\forall_{n\geq 2}(F_{\omega_n}\!\downarrow\,\rightarrow\rfn_{\isigma_n}(\Sigma_1))
\end{equation*}
of the claimed equivalence. The other direction is a straightforward formalization of Gentzen's lifting construction in \cite{gentzen43} (see also \cite[Theorem 4.11]{wainer-fairtlough-98}). Our approach to the remaining direction is essentially a formalization of \cite{buchholz-wainer-87} (refined in \cite{rathjen-carnielli-91}), heavily inspired by \cite[Appendix A]{rathjen13}. Let us describe this in some more detail: In \cite{buchholz91} one finds a rigorous proof that primitive recursive arithmetic can deduce the $\Pi_2$-reflection principle over Peano arithmetic from the assumption that no primitive recursive sequence of ordinals $<\varepsilon_0$ can descend infinitely long. Now the computation of a value $F_\alpha(n)$ is closely linked to certain descending sequences of ordinals below $\alpha$ (see \cite[Section 5.2]{sommer95} and \cite[Lemma 2.4]{rathjen13}). The theory $\isigma_1$ shows that these sequences must reach zero if $F_\alpha(n)$ is to be defined. Since this does not (immediately) apply to any primitive recursive sequence the result of \cite{buchholz91} is not sufficient for our purpose. Instead we work with the infinite proof system from \cite{buchholz-wainer-87}, featuring a particularly careful assignment of ordinals to proofs, closely liked to the computation of the values $F_\alpha(n)$. Now \cite{buchholz-wainer-87} is not concerned with strict bounds for fragments of Peano Arithmetic. The modifications which are necessary to get such bounds can be found in \cite{rathjen-carnielli-91}. The task is then to apply the formalization method of \cite{buchholz91} to the relevant parts of \cite{buchholz-wainer-87} and \cite{rathjen-carnielli-91}.

\section{The Finite Proof System and the Theories \texorpdfstring{$\isigma_n$}{ISigma-n}}

We want to analyze finite proofs in the theories $\isigma_n$ through an embedding into infinite proof figures. If we try to keep the infinite system as concise as possible then the embedding tends to get more difficult. To facilitate it we start with some transformations of the finite proof systems, assimilating them to features of the infinite system.\\
The first adaptation concerns the proof calculus itself. There is no canonic choice of proof calculus to be used with Peano Arithmetic, and it is standard that all calculi in use are ($\isigma_1$-provably) equivalent. Our choice is a Tait version of sequent calculus: This means that the relation symbols occur in pairs, with one member denoting the complement of the other. If $R$ is a name (in the meta-language) for a relation symbol then $\overline R$ denotes the complement of this relation symbol. Note that $\overline{\overline R}$ and $R$ refer to the same relation symbol of the object language. Often one member of the pair will have a natural name, like $\add(\cdot,\cdot,\cdot)$ for the graph of addition. The symbol for the complement will then be denoted by $\overline\add$. However, this does not mean that the relation symbol denoted by $\overline\add$ is less ``primitive'' than the relation symbol denoted by $\add$. Any relation symbol applied to the appropriate number of terms is a prime formula. Complex formulas are built by the connectives $\land$ and $\lor$ and the quantifiers $\forall$ and $\exists$. The negation symbol $\neg$ occurs only in the meta-language. Its use is governed by the stipulations that
\begin{align*}
 \neg R(t_1,\dots ,t_k)\quad&\text{is meta-notation for the formula}\quad\overline R(t_1,\dots ,t_k),\\
 \neg(A\land B)\quad&\text{denotes the same formula as}\quad(\neg A)\lor(\neg B),\\
 \neg(A\lor B)\quad&\text{denotes the same formula as}\quad(\neg A)\land(\neg B),\\
 \neg(\forall_x A)\quad&\text{denotes the same formula as}\quad\exists_x(\neg A),\\
 \neg(\exists_x A)\quad&\text{denotes the same formula as}\quad\forall_x(\neg A).
\end{align*}
As in \cite{buchholz-wainer-87} we always allow the following \textbf{rules of finite sequent calculus} (they may be complemented by additional rules specific to a certain theory):
\begin{gather*}
\begin{bprooftree}
\AxiomC{$\Gamma,A$}
\AxiomC{$\Gamma,B$}
\LeftLabel{($\land$)}
\BinaryInfC{$\Gamma,A\land B$}
\end{bprooftree}\qquad
\begin{bprooftree}
\AxiomC{$\Gamma,A$}
\LeftLabel{($\lor0$)}
\UnaryInfC{$\Gamma,A\lor B$}
\end{bprooftree}\qquad
\begin{bprooftree}
\AxiomC{$\Gamma,B$}
\LeftLabel{($\lor1$)}
\UnaryInfC{$\Gamma,A\lor B$}
\end{bprooftree}\\
\begin{bprooftree}
\AxiomC{$\Gamma,A$}
\LeftLabel{($\forall$)}
\RightLabel{($x$ not free in $\Gamma$)}
\UnaryInfC{$\Gamma,\forall_x A$}
\end{bprooftree}\qquad
\begin{bprooftree}
\AxiomC{$\Gamma,A[x:=t]$}
\LeftLabel{($\exists$)}
\UnaryInfC{$\Gamma,\exists_x A$}
\end{bprooftree}\qquad
\begin{bprooftree}
\AxiomC{$\Gamma,A$}
\AxiomC{$\Gamma,\neg A$}
\LeftLabel{(cut)}
\BinaryInfC{$\Gamma$}
\end{bprooftree}
\end{gather*}
We remark that $A[x:=t]$ denotes the result of substituting the term $t$ for the variable $x$ in the formula $A$ (avoiding variable clashes). With the exception of the eigenvariable $x$ of the rule ($\forall$) we require that any free variable in the lower sequent of a rule occurs in one of its upper sequents (cf.\ \cite{buchholz91}). The only cases where this is not automatic are the rules ($\exists$), because of variables that may occur in the term $t$, and the rule (cut). In these cases the superfluous variables in the upper sequent can be replaced by the constant $0$ (zero), which will always be part of the language.\\ 
Axioms are the initial sequents with which a proof starts. For the elimination of cuts it is important that any substitution instance (with terms substituted for free variables) of an axiom is an axiom. In all theories we allow the \textbf{logical axioms}
\begin{equation*}
 A,\neg A
\end{equation*}
for all atomic formulas $A$. Concerning side formulas, we allow that axioms and lower sequents of rules may be weakened at any moment, without that a weakening rule would have to be applied explicitly (following \cite{buchholz91}).\\
The next adaptation concerns the language of Peano Arithmetic. Here our choice is certainly not canonic: We follow \cite{buchholz91} in discarding all function symbols but $0$ and $S$ (successor), which will have the advantage that all closed terms are numerals and that we can deduce the conclusion of induction rules for arbitrary closed terms. For the moment we admit the relation symbols $=$ (with complement $\neq$) and $<$ (complement $\nless$), as well as ternary relations $\add(x,y,z)$ and $\mult(x,y,z)$ to denote the graphs of addition and multiplication. Later we will add more relation symbols to represent bounded formulas. All fragments of Peano Arithmetic will share (all substitution instances and weakenings of) the following \textbf{elementary axioms}:
\begin{gather*}\renewcommand{\arraystretch}{1.2}
x=x\\
x\neq y,y=x\\
x\neq y,y\neq z,x=z\\
Sx\neq 0\\
Sx\neq Sy,x=y\\
x\neq y,Sx=Sy\\
\add(x,0,x)\\
\overline\add(x,y,z),\add(x,Sy,Sz)\\
\overline\add(x,y,z),\overline\add(x,y,z'),z=z'\\
x\neq x',y\neq y',z\neq z',\overline\add(x,y,z),\add(x,y,z)\\
\mult(x,0,0)\\
\overline\mult(x,y,w),\overline\add(w,x,z),\mult(x,Sy,z)\\
\overline\mult(x,y,z),\overline\mult(x,y,z'),z=z'\\
x\neq x',y\neq y',z\neq z',\overline\mult(x,y,z),\mult(x,y,z)\\
x\nless 0\\
x\nless Sy,x<y,x=y\\
0<Sx\\
x\nless y,Sx<Sy\\
Sx\nless y,x<y\\
x<Sx\\
x\neq x',y\neq y',x\nless y,x'< y'
\end{gather*}
More familiar versions of the axioms are easily deduced, for example by the derivation
\begin{equation*}
\begin{bprooftree}
\AxiomC{$Sx\neq Sy,x=y$}
\LeftLabel{($\lor1$)}
\UnaryInfC{$Sx\neq Sy,Sx\neq Sy\lor x=y$}
\LeftLabel{($\lor0$)}
\UnaryInfC{$Sx\neq Sy\lor x=y$}
\LeftLabel{($\forall$)}
\UnaryInfC{$\forall_y(Sx\neq Sy\lor x=y)$}
\LeftLabel{($\forall$)}
\UnaryInfC{$\forall_x\forall_y(Sx\neq Sy\lor x=y)$}
\end{bprooftree}
\end{equation*}
If we add some amount of induction then the axioms characterizing inequality become certainly redundant. It will, however, be convenient to have these axioms from the start on, and certainly it does no harm. The alternative characterization of inequality by the formulas
\begin{gather*}
 \forall_{x,y}(x\nless y\lor\exists_z\add(Sx,z,y)),\\
 \forall_{x,y}(\forall_z\overline\add(Sx,z,y)\lor x< y)
\end{gather*}
will be provable in the theories $\isigma_n'$ with $n\geq 1$. To describe these theories we have to define the arithmetical hierarchy (we use primed symbols because we will introduce a different variant later):
\begin{itemize}
\item $\Delta_0'=\Sigma_0'=\Pi_0'$ is the class of bounded formulas, built from atomic formulas by the connectives and bounded quantifiers $\forall_{x< t}A\equiv\forall_x(x< t\rightarrow A)$ or $\exists_{x< t}A\equiv\exists_x(x< t\land A)$, where $t$ is a term that does not contain the bound variable $x$.
\item $\Sigma_{n+1}'$ is the class of formulas $\exists_x A$ where $A$ is a formula in the class $\Pi_n'$.
\item $\Pi_{n+1}'$ is the class of formulas $\forall_x A$ where $A$ is a formula in the class $\Sigma_n'$.
\end{itemize}
For $n\geq 1$ the theory $\isigma_n'$ consists of the logical axioms and the elementary axioms. Besides the general rules of finite sequent calculus it allows applications of the induction rule
\begin{equation*}
\begin{bprooftree}
\AxiomC{$\Gamma,A(0)$}
\AxiomC{$\Gamma,\neg A(x),A(Sx)$}
\LeftLabel{(Ind)}
\RightLabel{($x$ not free in $\Gamma$)}
\BinaryInfC{$\Gamma,A(t)$}
\end{bprooftree}
\end{equation*}
for any $\Sigma_n'$-formula $A$ and any term $t$. The advantage over induction axioms is that the formula $A(t)$ in the lower sequent has complexity $\Sigma_n'$. Partial cut elimination (in the finite system) will then allow to reduce all cut formulas to this complexity. As described in (cf.\ \cite{wainer-fairtlough-98}) the induction axioms can be deduced thanks to the possibility of side formulas in the induction rule:\\
\begin{equation*}
\begin{bprooftree}
\AxiomC{$\neg A(0),A(0)$}
\AxiomC{$\neg A(x),A(x)$}
\AxiomC{$A(Sx),\neg A(Sx)$}
\LeftLabel{($\land$)}
\BinaryInfC{$\neg A(x),A(Sx),A(x)\land\neg A(Sx)$}
\LeftLabel{($\exists$)}
\UnaryInfC{$\neg A(x),A(Sx),\exists_x(A(x)\land\neg A(Sx))$}
\LeftLabel{(Ind)}
\BinaryInfC{$\neg A(0),\exists(A(x)\land\neg A(Sx)),A(y)$}
\LeftLabel{($\forall$)}
\UnaryInfC{$\neg A(0),\exists(A(x)\land\neg A(Sx)),\forall_y A(y)$}
\end{bprooftree}
\end{equation*}
The reason for which we have required $n\geq 1$ is that the case $n=0$ does not give rise to the usual theory $\idelta_0'$: It cannot even prove the totality of addition. On the other hand, $\isigma_n'$ proves the formula $\forall_x\forall_y\exists_z\add(x,y,z)$ by induction on $y$ with the parameter $x$ fixed; the analogue is true for multiplication. It is a standard fact that one can then add function symbols $+$ and $\cdot$ and the axioms $\add(x,y,x+y)$ and $\mult(x,y,x\cdot y)$. The recursive clauses in their usual form are easily deduced. It remains to make induction available for formulas that contain the new function symbols: To this end one gives a translation that eliminates the function symbols from $\Sigma_n'$-formulas. The equivalence between a formula and its translation is provable in the system that contains the function symbols but not the new induction axioms (see \cite[Section I.2(e)]{hajek91} for details). Note that the translation which eliminates the function symbols does not only preserve provability but also truth. This justifies us in proving reflection principles for the theories without the function symbols.\\
We want to change the finite proof system yet another time. The purpose is to lift the logic of bounded quantifiers into the axioms, thus hiding it from the infinite proof system. To do so we introduce \textbf{relation symbols for bounded formulas}: Whenever $A$ is a $\Delta_0'$-formula in the old language and $(x_1,\dots ,x_k)$ lists the free variables of $A$ in some order then $R_A^{(x_1,\dots,x_k)}$ is such a relation symbol. In writing $R_A^{(x_1,\dots,x_k)}$ we implicitly demand that these conditions are met. The relation symbols $R_A^{(x_1,\dots,x_k)}$ and $R_{A[x_1:=y_1,\dots ,x_k:=y_k]}^{(y_1,\dots,y_k)}$ (simultaneous substitution) are identified. The idea is that the new atomic formula $R_A^{(x_1,\dots,x_k)}(t_1,\dots ,t_k)$ corresponds to the old $\Delta_0'$-formula $A[x_1:=t_1,\dots,x_k:=t_k]$. Accordingly, we define the extension of $R_A^{(x_1,\dots,x_k)}$ to be the set
\begin{equation*}
 \{(n_1,\dots,n_k)\,|\,A[x_1:=\overline{n_1},\dots,x_k:=\overline{n_k}]\text{ is true in the standard model}\}.
\end{equation*}
Back on the syntactic level, negation is defined by stating that
\begin{equation*}
 R_{\neg A}^{\vec x}\qquad\text{is the complement of }R_A^{\vec x}.
\end{equation*}
In addition to the logical axioms the new relation symbols are governed by the following \textbf{axioms of bounded logic} (remarks on the notation follow):
\begin{gather*}
 \neg A,R_A^{\vec x} \vec x\qquad\text{($A$ a prime formula)}\\
 \neg R_{A}^{\vec x} \vec x,\neg R_{B}^{\vec y} \vec y,R_{A\land B}^{\vec z} \vec z\\
 \neg R_{A}^{\vec x} \vec x,R_{A\lor B}^{\vec z} \vec z\\
 \neg R_{B}^{\vec y} \vec x,R_{A\lor B}^{\vec z} \vec z\\
 \neg R_{A[v:=t]}^{\vec x} \vec x,R_{\exists_v A}^{\vec y} \vec y\\
 R_{\forall_{v<0}A}^{\vec x} \vec x\\
 \neg R_{\forall_{v<w}A}^{\vec x} \vec x,\neg R_{A[v:=w]}^{\vec y} \vec y,R_{\forall_{v<Sw}A}^{\vec z} \vec z\\
 R_{\neg A}^{(x_1,\dots ,x_k)}(t_1,\dots ,t_k),R_{A[x_1:=t_1,\dots ,x_k:=t_k]}^{\vec y} \vec y
\end{gather*}
It may seem slightly irritating that the variables in the superscripts of the relation symbols correspond to the variables to which the relation symbols are applied. This is simply economic notation to assure the correct interdependencies: Consider, for example, an axiom of the second type in the list (conjunction introduction). Our notational convention assures that $\vec z$ lists the free variables of the formula $A\land B$. Thus $\vec z$ contains all variables listed in $\vec x$ and $\vec y$, in some order. The formulation of the axiom assures that the arguments of the relation symbols are related in the same way. Once the permutations of the arguments are understood one can, as always, pass on to arbitrary substitution instances. Of course, substituting into the formula $R_A^{\vec x} \vec x$ will not change the variables that appear in the (compound) relation symbol $R_A^{\vec x}$ but only the arguments to which this relation symbol is applied. Concretely, the sequent
\begin{equation*}
 \neg R_{x_1=x_2}^{(x_2,x_1)}(x_2,x_1),\neg R_{x_2=x_3}^{(x_2,x_3)}(x_2,x_3),R_{x_1=x_2\land x_2=x_3}^{(x_3,x_1,x_2)}(x_3,x_1,x_2)
\end{equation*}
is an axiom. Thus its substitution instance
\begin{equation*}
 \neg R_{x_1=x_2}^{(x_2,x_1)}(y,x),\neg R_{x_2=x_3}^{(x_2,x_3)}(y,z),R_{x_1=x_2\land x_2=x_3}^{(x_3,x_1,x_2)}(z,x,y)
\end{equation*}
is an axiom as well. In view of $(x_1=x_2)[x_2:=y,x_1:=x]\equiv (x=y)$ and similar substitutions into the other formulas it corresponds to the sequent
\begin{equation*}
 x\neq y,y\neq z,x=y\land y=z.
\end{equation*}
On a less technical note, observe that we have axioms to introduce the propositional connectives and the existential quantifier in the subscripts. We cannot give an axiom that introduces a universal quantifier, since the rule ($\forall$) does not preserve the truth of arbitrary substitution instances. To introduce bounded quantifiers we will use induction for the formula $R_{\forall_{v<y}A}^{\vec x,y} (\vec x,y)$ with induction variable $y$. Of course, in the induction step the variable $y$ only increases in argument position, not in the subscript. The necessary connection is provided by the last of the axioms of bounded logic.\\
One should convince oneself that any closed substitution instance of the axioms of bounded logic contains a true (prime) formula. Let us verify this for the last type of axioms (substitution of terms): A closed substitution instance of this axiom has the form
\begin{equation*}
 R_{\neg A}^{(x_1,\dots ,x_k)}(t_1[\vec y:=\vec n],\dots ,t_k[\vec y:=\vec n]),R_{A[x_1:=t_1,\dots ,x_k:=t_k]}^{\vec y} \vec n.
\end{equation*}
Thus we have to see that either the formula
\begin{equation*}
 \neg A[x_1:=t_1[\vec y:=\vec n],\dots,x_k:=t_k[\vec y:=\vec n]]
\end{equation*}
or the formula
\begin{equation*}
 (A[x_1:=t_1,\dots ,x_k:=t_k])[\vec y:=\vec n]
\end{equation*}
is true. By a standard fact about the composition of substitutions the first formula is indeed the negation of the second.\\
In the extended language, let us define a more restrictive variant of the \textbf{arithmetical hierarchy}:
\begin{itemize}
\item $\Delta_0=\Sigma_0=\Pi_0$ is the class of atomic formulas, i.e.\ formulas of the form $R(t_1,\dots,t_k)$ where $R$ is a relation symbol and $t_1,\dots,t_k$ are terms.
\item $\Sigma_{n+1}$ is the class of formulas $\exists_x A$ where $A$ is a formula in the class $\Pi_n$.
\item $\Pi_{n+1}$ is the class of formulas $\forall_x A$ where $A$ is a formula in the class $\Sigma_n$.
\end{itemize}
For $n\geq 1$ the theory $\isigma_n$ consists of the logical axioms, the elementary axioms and the axioms for bounded logic. Besides the general rules of finite sequent calculus it allows applications of the induction rule whenever the induction formula is in the class $\Sigma_n$.\\
We would like to reduce the theory $\isigma_n'$ to the theory $\isigma_n$, but before we should simplify the notation: Given a $\Delta_0'$-formula $A$, we write $\ul A$ for the formula $R_A^{\vec x} \vec x$ where the variables $\vec x$ are ordered according to their occurrence in $A$. Note that we cannot order the superscript according to an ``external principle'', as e.g.\ the indices of the variables. The issue becomes apparent in an example: Since $\forall_{x_0}\,x_0<x_1$ and $\forall_{x_2}\,x_2<x_1$ are the same formula, the formulas $\forall_{x_0}\ul{x_0<x_1}$ and $\forall_{x_2}\ul{x_2<x_1}$ should be equal as well. Now the formula $\forall_{x_0} R_{x_0<x_1}^{(x_0,x_1)}(x_0,x_1)$ is different from the formula $\forall_{x_2} R_{x_2<x_1}^{(x_1,x_2)}(x_1,x_2)$ (even if the two formulas are equivalent). On the other hand, $\forall_{x_0} R_{x_0<x_1}^{(x_0,x_1)}(x_0,x_1)$ and $\forall_{x_2} R_{x_2<x_1}^{(x_2,x_1)}(x_2,x_1)$ are the same formula (recall that we identify relation symbols that only differ in the names of the variables). In the new notation, the sequent
\begin{equation*}
\neg A,\ul A\qquad\text{(for $A$ a prime formula)}
\end{equation*}
is an axiom of bounded logic. One should observe that $\neg\ul A$ and $\ul{\neg A}$ are the same formula, so that the ``converse implication'' $A,\neg\ul A$ is an axiom of the same type. For any $\Delta_0'$-formulas $A$, $B$ and $C$ the following are also axioms of bounded logic:
\begin{gather*}
\neg\ul A,\neg\ul B,\ul{A\land B}\\
\neg\ul A,\ul{A\lor B}\\
\neg\ul B,\ul{A\lor B}\\
\neg\ul{t<y\land A[x:=t]},\ul{\exists_{x<y} A}\\
\ul{\forall_{v<0} A}\\
\neg\ul{\forall_{x<y} A},\neg\ul{A[x:=y]},\ul{\forall_{x<Sy} A}\\
\ul{\neg A}[x:=t],\ul{A[x:=t]}
\end{gather*}
We will permit ourselves to use the admissible rules associated with these axioms. Thus e.g.
\begin{equation*}
 \begin{bprooftree}
 \AxiomC{$\Gamma,\ul A$}
 \AxiomC{$\Gamma,\ul B$}
 \BinaryInfC{$\Gamma,\ul{A\land B}$}
 \end{bprooftree}
\end{equation*}
alludes to the following proof tree:
\begin{equation*}
\begin{bprooftree}
\AxiomC{$\Gamma,\ul B$}
\AxiomC{$\Gamma,\ul A$}
\AxiomC{$\neg\ul A,\neg\ul B,\ul{A\land B}$}
\LeftLabel{(cut)}
\BinaryInfC{$\Gamma,\neg\ul B,\ul{A\land B}$}
\LeftLabel{(cut)}
\BinaryInfC{$\Gamma,\ul{A\land B}$}
\end{bprooftree}
\end{equation*}
The other axioms give rise to admissible rules that allow us to box or un-box a prime formula, to introduce a disjunction inside a box, to introduce an existential quantifier inside a box, and to pull a substitution inside or out of a box. Recall also that negation is a notion of the meta-laguage, and that $\neg\ul A$ and $\ul{\neg A}$ are the same formula. Thus, moving a negation inside or out of a box does not even require an admissible rule.\\
The notation $\ul A$ can be extended to arbitrary formulas $A$ of the old language (i.e.\ without occurrences of the relation symbols $R_A^{\vec x}$): To do so, one states that $\ul{B\land C}$ (resp.\ $\forall_x B$, and analogously for $\lor$ and $\exists$) is the same formula as $\ul B\land\ul C$ (resp.\ $\forall_x \ul B$) if $B\land C$ (resp.\ $\forall_x B$) is not a $\Delta_0'$-formula. Equivalently, one could say that we replace maximal $\Delta_0'$-subformulas. Recall that we want to reduce the theory $\isigma_n'$ to the theory $\isigma_n$. The crucial step is to see that $\mathbf{I\Delta}_0$ admits the boxing and un-boxing of arbitrary formulas:
\begin{equation*}
 \parbox{11cm}{For any formula $A$ there is an $\mathbf{I\Delta}_0$-proof of the sequent $\neg A,\ul A$.}
\end{equation*}
This is proved by structural induction over the formula $A$. For a prime formula the required sequent is an axiom of bounded logic. The cases of a conjunction, a disjunction and an existential quantifier are easy: One uses the induction hypothesis and the rules of sequent calculus, or the admissible rules where one is concerned with $\Delta_0'$-formulas. Observe also that $A$ and $\ul A$ have the same free variables. Now assume that $A$ is a universal formula. The case where $A$ is not a $\Delta_0'$-formula is again easy. Otherwise we have $A\equiv\forall_{v<t}B$ for some $\Delta_0'$-formula $B$. The crucial idea is to prove $\neg A,\ul{w\nleq t\lor\forall_{v<w}B}$ by induction over the fresh variable $w$. Of course, the induction has to take place outside the box. This is easy to achieve because substitutions can be pulled in and out. Let us first display the end of the proof, deducing the desired end-sequent from the result of the induction (the left initial sequent):
\begin{equation*}
\begin{bprooftree}
\AxiomC{$\neg A,\ul{w\nleq t\lor\forall_{v<w}B}[w:=t]$}
\UnaryInfC{$\neg A,\ul{t\nleq t\lor\forall_{v<t}B}$}
\AxiomC{$t\leq t$}
\UnaryInfC{$\ul{t\leq t}$}
\AxiomC{$\ul{\neg\forall_{v<t}B},\ul{\forall_{v<t}B}$}
\BinaryInfC{$\ul{t\leq t\land\neg\forall_{v<t}B},\ul{\forall_{v<t}B}$}
\LeftLabel{(cut)}
\BinaryInfC{$\neg A,\ul{\forall_{v<t}B}$}
\end{bprooftree}
\end{equation*}
Since the induction formula is boxed it is indeed a $\Delta_0$-formula. Recall that the induction variable $w$ was fresh, so that it does not appear in the side formula $\neg A$. Thus it suffices to deduce the two premises of the induction rule, namely the base
\begin{equation*}
 \neg A,\ul{w\nleq t\lor\forall_{v<w}B}[w:=0]
\end{equation*}
and the step
\begin{equation*}
 \neg A,\neg\ul{w\nleq t\lor\forall_{v<w}B},\ul{w\nleq t\lor\forall_{v<w}B}[w:=Sw].
\end{equation*}
Using the admissible rules, the base case can be reduced to $\ul{0\nleq t\lor\forall_{v<0}B}$ and then to $\ul{\forall_{v<0}B}$, which is an axiom. For the step one observes that $\neg\ul{w\nleq t\lor\forall_{v<w}B}$ is the same formula as $\ul{w\leq t\land\neg\forall_{v<w}B}$. Since we can introduce a conjunction in a box and in view of the proof
\begin{equation*} 
\begin{bprooftree}
\AxiomC{$w\leq t,Sw\nleq t$}
\UnaryInfC{$\ul{w\leq t},\ul{Sw\nleq t}$}
\UnaryInfC{$\ul{w\leq t},\ul{(w\nleq t\lor\forall_{v<w}B)[w:=Sw]}$}
\UnaryInfC{$\ul{w\leq t},\ul{(w\nleq t\lor\forall_{v<w}B)}[w:=Sw]$}
\end{bprooftree}
\end{equation*}
it remains to deduce the sequent
\begin{equation*}
 \neg A,\ul{\neg\forall_{v<w}B},\ul{w\nleq t\lor\forall_{v<w}B}[w:=Sw].
\end{equation*}
This sequent can be proved as follows. Note that the rule marked (Ax) arises from an axiom of bounded logic. The right initial sequent comes from the induction hypothesis:
\begin{equation*}
\hspace{-1cm}
\scalebox{.9}{
\begin{bprooftree}
\AxiomC{$\neg\ul{\forall_{v<w}B},\ul{\forall_{v<w}B}$}
\AxiomC{$w<t\land\neg B[v:=w],Sw\nleq t,B[v:=w]$}
\LeftLabel{($\exists$)}
\UnaryInfC{$\neg A,Sw\nleq t,B[v:=w]$}
\AxiomC{$\neg B[v:=w],\ul B[v:=w]$}
\BinaryInfC{$\neg A,Sw\nleq t,\ul B[v:=w]$}
\UnaryInfC{$\neg A,Sw\nleq t,\ul{B[v:=w]}$}
\LeftLabel{(Ax)}
\BinaryInfC{$\neg A,Sw\nleq t,\neg\ul{\forall_{v<w}B},\ul{\forall_{v<Sw}B}$}
\UnaryInfC{$\neg A,\ul{Sw\nleq t},\neg\ul{\forall_{v<w}B},\ul{\forall_{v<Sw}B}$}
\UnaryInfC{$\neg A,\neg\ul{\forall_{v<w}B},\ul{(w\nleq t\lor\forall_{v<w}B)[w:=Sw]}$}
\UnaryInfC{$\neg A,\neg\ul{\forall_{v<w}B},\ul{(w\nleq t\lor\forall_{v<w}B)}[w:=Sw]$}
\end{bprooftree}}
\end{equation*}
This finishes the induction. The established result entails a new admissible rule for the theory $\idelta_0$: Namely, we can now box and un-box arbitrary formulas. In fact, this is even allowed ``in the scope of substitution instructions'', for the sequent $\neg A[x:=t],\ul A[x:=t]$ is provable as well. It follows that the theory $\isigma_n$ is conservative over the theory $\isigma_n'$: Indeed, an $\isigma_n'$-proof is an $\isigma_n$-proof until we hit an induction rule
\begin{equation*}
\begin{bprooftree}
\AxiomC{$\Gamma,A[x:=0]$}
\AxiomC{$\Gamma,\neg A,A[x:=Sx]$}
\LeftLabel{(Ind)}
\RightLabel{($x$ not free in $\Gamma$)}
\BinaryInfC{$\Gamma,A[x:=t]$}
\end{bprooftree}
\end{equation*}
where $A$ is a $\Sigma_n'$-formula. Since $\ul A$ is a $\Sigma_n$-formula we can replace the rule by an induction rule of the theory $\isigma_n$, as follows:
\begin{equation*}
\begin{bprooftree}
\AxiomC{$\Gamma,A[x:=0]$}
\UnaryInfC{$\Gamma,\ul A[x:=0]$}
\AxiomC{$\Gamma,\neg A,A[x:=Sx]$}
\UnaryInfC{$\Gamma,\neg\ul A,\ul A[x:=Sx]$}
\LeftLabel{(Ind)}
\RightLabel{($x$ not free in $\Gamma$)}
\BinaryInfC{$\Gamma,\ul A[x:=t]$}
\UnaryInfC{$\Gamma,A[x:=t]$}
\end{bprooftree}
\end{equation*}
If the end-sequent of a proof consists of a single $\Sigma_1'$-formula then we can turn it into an equivalent $\Sigma_1$-formula by boxing it. The $\Sigma_1'$-reflection principle over the theory $\isigma_n'$ is thus reduced to the $\Sigma_1$-reflection principle over the theory $\isigma_n$.\\
Let us conclude with two minor changes to the finite proof system. The first concerns the cut rule: The free-cut elimination theorem (see e.g.\ \cite[Section 2.4.6]{buss-introduction-98}), which is provable in the theory $\isigma_1$, transforms a given $\isigma_n$-proof into an $\isigma_n$-proof where all cut formulas are prime formulas or in the class $\Sigma_n\cup\Pi_n$. We may just as well assume that any $\isigma_n$-proof has this property from the outset. Finally, when we combine two proofs by a binary rule it can of course happen that these proofs have different heights. We can avoid this if we fill the shorter proof with repetition rules of the following form:
\begin{equation*}
\begin{bprooftree}
\AxiomC{$\Gamma$}
\LeftLabel{(Rep)}
\UnaryInfC{$\Gamma$}
\end{bprooftree}
\end{equation*}
We can now give an official definition of the sets $\zet_n$ of \textbf{finite proofs} in the theories $\isigma_n$, with $n\geq 1$.

\begin{definition}\label{def:finite-proofs}
By ``formula" we mean a formula in the extended language above, i.e.\ containing the relation symbols $R_A^{\vec x}$. For each $n\geq 1$ we give an inductive definition of a set $\zet_n$ of terms typed by sequents of such formulas. The height of a term is the (finite) stage at which it appears in the inductive definition:
\begin{description}
\item[$(\ax)$] If the sequent $\Gamma$ contains a logical axiom, an elementary axiom or an axiom of bounded logic then $\ax^\Gamma$ is a term in $\zet_n$.
\item[$(\rep)$] If $\Gamma$ is a sequent and we have $d_0^{\Gamma_0}\in\zet_n$ with $\Gamma_0\subseteq\Gamma$ then $(\rep\, d_0^{\Gamma_0})^\Gamma$ is a term in $\zet_n$.
\item[$(\land)$] If the sequent $\Gamma$ contains the formula $A\equiv A_0\land A_1$ and we have $d_0^{\Gamma_0},d_1^{\Gamma_1}\in\zet_n$ of the same height with $\Gamma_0\subseteq\Gamma\cup\{A_0\}$ and $\Gamma_1\subseteq\Gamma\cup\{A_1\}$ then $(d_0^{\Gamma_0}\land_A d_1^{\Gamma_1})^\Gamma$ is a term in $\zet_n$.
\item[$(\lor)$] If the sequent $\Gamma$ contains the formula $A\equiv A_0\lor A_1$ and we have $d_0^{\Gamma_0}\in\zet_n$ with $\Gamma_0\subseteq\Gamma\cup\{A_i\}$ for some $i\in\{0,1\}$ then $(\lor_{i,A}d_0^{\Gamma_0})^\Gamma$ is a term in $\zet_n$.
\item[$(\forall)$] Assume that the sequent $\Gamma$ contains the formula $A\equiv\forall_x A_0(x)$, and that the variable $v$ is not free in $\Gamma$. If we have $d_0^{\Gamma_0}\in\zet_n$ with $\Gamma_0\subseteq\Gamma\cup\{A_0(v)\}$ then $(\forall_{v,A}d_0^{\Gamma_0})^\Gamma$ is a term in $\zet_n$.
\item[$(\exists)$] If the sequent $\Gamma$ contains the formula $A\equiv\exists_x A_0(x)$ and we have $d_0^{\Gamma_0}\in\zet_n$ with $\Gamma_0\subseteq\Gamma\cup\{A_0(t)\}$, where $t$ is some (object) term all variables of which are free in $\Gamma$, then $(\exists_{t,A}d_0^{\Gamma_0})^\Gamma$ is a term in $\zet_n$.
\item[$(\cut)$] If $\Gamma$ is a sequent and we have $d_0^{\Gamma_0},d_1^{\Gamma_1}\in\zet_n$ of the same height and with $\Gamma_0\subseteq\Gamma\cup\{\neg A\}$ and $\Gamma_1\subseteq\Gamma\cup\{A\}$, where $A$ is some formula in $\bigcup_{k\leq n}\Sigma_k$ all free variables of which appear in $\Gamma$, then $(d_0^{\Gamma_0}\scut_A d_1^{\Gamma_1})^\Gamma$ is a term in $\zet_n$.
\item[$(\ind)$] Assume that the sequent $\Gamma$ contains the $\Sigma_n$-formula $A[v:=t]$, that the variable $v$ is not free in $\Gamma$, and that all free variables of $t$ appear in $\Gamma$. If we have $d_0^{\Gamma_0},d_1^{\Gamma_1}\in\zet_n$ of the same height with $\Gamma_0\subseteq\Gamma\cup\{A[v:=0]\}$ and $\Gamma_1\subseteq\Gamma\cup\{\neg A,A[v:=Sv]\}$ then $(d_0^{\Gamma_0}\sind_{v,t,A} d_1^{\Gamma_1})^\Gamma$ is a term in $\zet_n$.
\end{description}
We write $\height(d)$ for the \textbf{height} of $d\in\zet_n$. The type $\Gamma$ of a term $d^\Gamma\in\zet_n$ is also called the \textbf{end-sequent} of $d$. We write $\zet_n^0$ for the set of those $d\in\zet_n$ which have closed end-sequent. The \textbf{term depth} $\dterm(d)$ of $d\in\zet_n$ is the maximal depth of an object term $t$ that appears as an index of a rule $\exists_{t,A}$ or $\sind_{v,t,A}$ used in the construction of $d$. The \textbf{cut-rank} $\dcut(d)$ of $d\in Z_n$ is $n$ if an induction rule is used in the construction of $d$. Otherwise it is the maximal $k\leq n$ such that the construction of $d$ involves a rule $\scut_A$ with $A\in\Sigma_k$. Given a proof $d^\Gamma\in\zet_n$, a variable $v$ and a numeral $m$ one can \textbf{substitute} appropriate occurrences of $v$ by $m$, to obtains a proof $d[v:=m]^{\Gamma[v:=m]}\in\zet_n$ with $\height(d[v:=m])=\height(d)$, $\dcut(d[v:=m])=\dcut(d)$ and $\dterm(d[v:=m])\leq\dterm(d)+m$.
\end{definition}

\section{The Infinite Proof System}

The infinite proof system we use is based on \cite{buchholz-wainer-87}. An important modification is due to \cite{rathjen-carnielli-91}: To make cut-elimination efficient for small ordinals one has to admit exponential shifts in the index of the accumulation rule. We will eventually formalize the infinite system in the style of \cite{buchholz91}: To model infinite proofs by finite terms one exhibits (primitive recursive) functions which describe the rule, the end-sequent, the ordinal tag, and the immediate subtrees (as terms) of the infinite proof tree associated with a term. Before we switch to term notations, let us present the modified infinite system in a more informal style: As usual, the sequents deduced by the infinite proof system are finite sets of closed formulas. Note that any closed term is a numeral. The formulas of the infinite system fall into two classes: First, we have the formulas of the finite proof systems (see Definition \ref{def:finite-proofs}). Additionally, we now allow the \textbf{special formulas} $n\in N$ and $n\notin N$, one of which is the negation of the other. We consider $n\in N$ (but not $n\notin N$) as a formula in the class $\Delta_0$. Note, however, that the special formulas may not appear as building blocks of compound formulas. In the meta-theory, we will eventually interpret $N$ as a finite set of natural numbers. One should think of all quantifiers as relativized to the set $N$. Next, we remark that the accumulation rule of \cite{buchholz-wainer-87} relies on a \textbf{step-down relation} between ordinals. We need to assume that any limit ordinal $\lambda<\varepsilon_0$ is associated with a \textbf{fundamental sequence}, i.e.\ a strictly increasing sequence $(\{\lambda\}(n))_{n\in\mathbb N}$ with supremum $\lambda$. The precise definition of fundamental sequences vary, and we adopt the version of \cite{buchholz-wainer-87}. It is shown in \cite{sommer95} that the ternary relation $\{\lambda\}(n)=\alpha$ can be defined by a $\Delta_0$-formula, and that the function $(\lambda,n)\mapsto\{\lambda\}(n)$ is $\idelta_0(\exp)$-provably total. It is convenient to extend the notation to successor ordinals and zero by setting $\{\alpha+1\}(n):=\alpha$ and $\{0\}(n):=0$. We write $\alpha\leq_k\beta$ if there is a sequence $(\alpha_0,\dots ,\alpha_n)$ of ordinals with $\beta=\alpha_0$, $\alpha=\alpha_n$ and $\alpha_{m+1}=\{\alpha_m\}(k)$ for all $m<n$. Slightly deviating from \cite{buchholz-wainer-87} we define
\begin{equation*}
k(\Gamma):=\max(\{1\}\cup\{n\,|\,\text{the formula $n\notin N$ appears in the sequent $\Gamma$}\}).
\end{equation*}
Somewhat unusually we set
\begin{equation*}
3_{0+}^k:=k\quad\text{and}\quad 3_{1+}^k:=3^{k+1}.
\end{equation*}
The accumulation rule can now be given as
\begin{equation*}
 \begin{bprooftree}
\AxiomC{$\vdash^\alpha\Gamma$}
\LeftLabel{($\acc^i$)}
\RightLabel{(if we have $\alpha+1\leq_{ 3_{+i}^{k(\Gamma)}}\beta$, with $i\in\{0,1\}$)}
\UnaryInfC{$\vdash^\beta\Gamma$}
\end{bprooftree}
\end{equation*}
Note that we have $\alpha+1$ where \cite{buchholz-wainer-87} has $\alpha$ (and then $<_{k(\Gamma)}$ instead of $\leq_{k(\Gamma)}$). This will slightly simplify the bounding lemma, because it leads to \emph{strict} inequalities between certain fast-growing functions. Next, let us give the propositional rules and the cut rule. Compared with \cite{buchholz-wainer-87}, each of our rules is followed by an implicit application of accumulation. Otherwise, most proofs would end with an explicit accumulation rule, which would complicate the formalization.
\begin{gather*}\addtolength{\jot}{5pt}
\begin{bprooftree}
\AxiomC{$\vdash^\alpha\Gamma,A_j$}
\LeftLabel{($\lor_{j,A_0\lor A_1}^i$)}
\RightLabel{(if we have $\alpha+1\leq_{ 3_{i+}^{k(\Gamma,A_0\lor A_1)}}\beta$)}
\UnaryInfC{$\vdash^\beta\Gamma,A_0\lor A_1$}
\end{bprooftree}\\
\begin{bprooftree}
\AxiomC{$\vdash^\alpha\Gamma,A_0$}
\AxiomC{$\vdash^\alpha\Gamma,A_1$}
\LeftLabel{($\land_{A_0\land A_1}^i$)}
\RightLabel{(if we have $\alpha+1\leq_{ 3_{i+}^{k(\Gamma,A_0\land A_1)}}\beta$)}
\BinaryInfC{$\vdash^\beta\Gamma,A_0\land A_1$}
\end{bprooftree}\\
\begin{bprooftree}
\AxiomC{$\vdash^\alpha\Gamma,\neg A$}
\AxiomC{$\vdash^\alpha\Gamma,A$}
\LeftLabel{($\cut_A^i$)}
\RightLabel{(if we have $\alpha+1\leq_{ 3_{i+}^{k(\Gamma)}}\beta$, and $A\in\Sigma_k$ for some $k$)}
\BinaryInfC{$\vdash^\beta\Gamma$}
\end{bprooftree}
\end{gather*}
The quantifier rules need to be modified in comparison with \cite{buchholz-wainer-87}, because we view all quantifiers as relativized to the special symbol $N$, which will be necessary to invert on $\Sigma_1$-cuts. Note in particular the summand $2$ in the existential rule. It is necessary because inverting on a relativized universal quantifier adds two formulas (the formula which relativizes the quantified variable, in addition to the matrix of the universal formula), so in the reduction lemma we will need room for two cuts.
{\addtolength{\jot}{5pt}
\begin{gather*}
\begin{bprooftree}
\AxiomC{$\vdash^\alpha\Gamma,m\in N$}
\AxiomC{$\vdash^\alpha\Gamma,A(m)$}
\LeftLabel{($\exists_{m,\exists_xA(x)}^i$)}
\RightLabel{(if we have $\alpha+2\leq_{3_{i+}^{k(\Gamma,\exists_x A(x))}}\beta$)}
\BinaryInfC{$\vdash^\beta\Gamma,\exists_x A(x)$}
\end{bprooftree}\\
\begin{bprooftree}
\AxiomC{$\{\vdash^\alpha\Gamma,m\notin N,A(m)\}_{m\in\mathbb N}$}
\LeftLabel{($\omega_{\forall_x A(x)}^i$)}
\RightLabel{(if we have $\alpha+1\leq_{3_{i+}^{k(\Gamma,\forall_x A(x))}}\beta$)}
\UnaryInfC{$\vdash^\beta\Gamma,\forall_x A(x)$}
\end{bprooftree}
\end{gather*}}
Finally, we have the \textbf{truth axioms}
\begin{equation*}
\vdash^\alpha A\text{ (for $A$ a true prime formula not of the form $n\in N$ or $n\notin N$)}
\end{equation*}
and the \textbf{$N$-axioms}
\begin{equation*}
\vdash^\alpha\Gamma,n\in N,n\notin N\quad\text{and}\quad\vdash^\alpha\Gamma,n\notin N,(n+1)\in N
\end{equation*}
with an arbitrary ordinal $\alpha\geq 2$.\\
Our next goal is to build finite term notations which model (the necessary occurrences of) these rules, as in \cite{buchholz91}. Before we do so we will, however, develop a different viewpoint on the step-down relation $\alpha\leq_k\beta$. These \textbf{step-down arguments} do not seem to be strictly necessary, but they are nice because they give the proof system a more syntactic flavour:

\begin{definition}
Step-down arguments are terms typed by pairs of ordinals. They are inductively defined as follows:
\begin{description}
\item[$(\fund)$] Given $\alpha=\{\beta\}(m)$, the expression $\fund(m)^\beta_\alpha$ is a step-down argument.
\item[$(*)$] If $s^\beta_\alpha$ and $s^\gamma_\beta$ are step-down arguments then the expression $(s^\beta_\alpha*s^\gamma_\beta)^\gamma_\alpha$ is a step-down argument.
\item[$(+)$] Consider a step-down argument $s^\gamma_\beta$ and an ordinal $\alpha$ which meshes with $\gamma$ (i.e.\ we have $\alpha=0$ or else $\gamma<\omega^{\alpha_0+1}$ where $\alpha_0$ is the smallest exponent in the Cantor normal form of $\alpha$, see e.g.\ \cite[Definition 2.6]{rathjen13}). Then the expression $(\alpha+s^\gamma_\beta)^{\alpha+\gamma}_{\alpha+\beta}$ is a step-down argument.
\item[($\omega$)] If $s^\beta_\alpha$ is a step-down argument then so is $(\omega^{s^\beta_\alpha})^{\omega^\beta}_{\omega^\alpha}$.
\end{description}
The \textbf{base} $\stepba(s)$ of a step-down argument $s$ is the maximal number $m$ such that an expression of the form $\fund(m)^\beta_\alpha$ occurs in the construction of $s$. To read off the type of a step-down argument we use the notation $\stepto(s^\beta_\alpha)=\beta$ and $\stepbo(s^\beta_\alpha)=\alpha$. We introduce the following abbreviation:
\begin{equation*}
 s\vdash\alpha\leq_k\beta\quad:\Leftrightarrow\quad\stepto(s)=\beta\land\stepbo(s)=\alpha\land\stepba(s)\leq k
\end{equation*}
To express that $s\vdash\alpha\leq_k\beta$ holds we also say that $s$ is a step-down argument for $\alpha\leq_k\beta$.\\
When working with step-down arguments, we will omit ordinal tags which can be inferred. For example, $\fund(m)*\fund(m)^\alpha$ will stand for the step-down argument $(\fund(m)^{\{\alpha\}(m)}_{\{\{\alpha\}(m)\}(m)}*\fund(m)^\alpha_{\{\alpha\}(m)})^\alpha_{\{\{\alpha\}(m)\}(m)}$. For definiteness, let us agree that $s_0*s_1*s_2$ is bracketed as $(s_1*s_2)*s_3$. Only for the construct $\fund(\cdot)$ we use exponential notation to denote iterations: Thus, $\fund(m)*\fund(m)^\alpha$ is further abbreviated as $\fund^2(m)^\alpha$.
\end{definition}

All rules of the infinite system should now be reformulated in terms of step-down arguments. For example the introduction of a conjunction gets the following form:
\begin{equation*}
\begin{bprooftree}
\AxiomC{$\vdash^\alpha\Gamma,A_0$}
\AxiomC{$\vdash^\alpha\Gamma,A_1$}
\LeftLabel{($\land_{A_0\land A_1}^{i,s}$)}
\RightLabel{(if $s\vdash\alpha+1\leq_{3_{i+}^{k(\Gamma,A_0\land A_1)}}\beta$)}
\BinaryInfC{$\vdash^\beta\Gamma,A_0\land A_1$}
\end{bprooftree}
\end{equation*}
All other rules are modified in the same way. We need some properties of step-downs:

\begin{lemma}\label{lem:step-down-properties}
We have the following basic facts about step-down arguments:
\begin{enumerate}[label=(\alph*)]
\item From $s\vdash\alpha\leq _k\beta$ and $k'\geq k$ we can infer $s\vdash\alpha\leq_{k'}\beta$.
\item If there is a step-down argument $s$ with $\stepto(s)=\beta$ and $\stepbo(s)=\alpha$ then we have $\alpha\leq\beta$.
\item If there is a step-down argument $s$ with $\stepto(s)=\omega\cdot p+q$, $\stepbo(s)=\omega\cdot p'+q'$ and $\stepba(s)\leq k$ then we have either $p'=p$ and $q'\leq q$ or else $p'<p$ and $q'\leq k+1$.
\end{enumerate}
\end{lemma} 
We remark that (a) will ensure that the infinite proof system behaves well with respect to weakening (cf.\ the comment directly before \cite[Definition 5]{buchholz-wainer-87}).
\begin{proof}
Part (a) follows immediately from the definition. Parts (b) and (c) are shown by structural induction on step-down arguments (with fixed $k$ in the case of (c)). Concerning (b), the case $(\fund)$ holds because fundamental sequences approximate ordinals from below. The steps follow from the transitivity of the ordering and the monotonicity of the functions $\gamma\mapsto\alpha+\gamma$ and $\gamma\mapsto\omega^\gamma$. Coming to (c), note first that by (b) we can only have $\omega\cdot p+q=0$ if we also have $\omega\cdot p'+q'=0$, in which case the claim is true. So assume now $\omega\cdot p+q>0$. The case $s=\fund(m)^{\omega\cdot p+q}_{\omega\cdot p'+q'}$ with $m=\stepba(s)\leq k$ holds in view of
\begin{equation*}
\{\omega\cdot p+q\}(m)=\begin{cases}
\omega\cdot p +(q-1) & \text{if $q>0$},\\
\omega\cdot (p-1)+(m+1) & \text{otherwise}.
\end{cases}
\end{equation*}
Next, assume that $s$ is of the form $s_0*s_1$. Using (b) we have
\begin{equation*}
 \stepbo(s_1)\leq\stepto(s_1)=\stepto(s)<\omega^2,
\end{equation*}
so that $\stepto(s_0)=\stepbo(s_1)$ is also of the form $\omega\cdot p''+q''$. Since $\stepba(s_0)$ and $\stepba(s_1)$ are both bounded by $\stepba(s)$ we can apply the induction hypothesis to $s_0$ and $s_1$. This leads to four different cases, all of which are easily checked.\\
Now assume that $s$ is of the form $\alpha+s_0$. First, $\alpha+\stepto(s_0)=\stepto(s)=\omega\cdot p+q$ implies that $\alpha$ and $\stepto(s_0)$ are of the form $\alpha=\omega\cdot p_\alpha+q_\alpha$ and $\stepto(s_0)=\omega\cdot p_0+q_0$. Since $\alpha$ and $\stepto(s_0)$ must mash we have $p=p_\alpha+p_0$ and $q=q_\alpha+q_0$, and either $q_\alpha=0$ or $p_0=0$. Using (b) we learn that $\stepbo(s_0)$ is also of the form $\beta=\omega\cdot p_1+q_1$ and that $\alpha$ mashes with $\stepbo(s_0)$, which implies $p'=p_\alpha+p_1$ and $q'=q_\alpha+q_1$. Furthermore the induction hypothesis tells as that we have either $p_1=p_0$ and $q_1\leq q_0$ or $p_1< p_0$ and $q_1\leq k+1$. This leaves us with four cases to check: If e.g.\ we have $q_\alpha=0$ as well as $p_1=p_0$ and $q_1\leq q_0$, then we can conclude $p'=p_\alpha+p_1=p_\alpha+p_0=p$ and $q'=q_1\leq q_0=q$. The other cases are left to the reader.\\
Finally, assume that $s$ is of the form $\omega^{s_0}$. Then $\omega^{\stepto(s_0)}=\stepto(s)<\omega^2$ implies $\stepto(s_0)\leq 1$. In view of (b) the only non-trivial case amounts to $\stepto(s_0)=1$ and $\stepbo(s_0)=0$. There we have $\stepto(s)=\omega$ and $\stepbo(s)=1$, i.e.\ $p'=0<1=p$ and $q'=1\leq k+1$.
\end{proof}

We will need the following constructions of step-downs:

\begin{lemma}\label{lem:transformations-step-downs}
The following constructions and transformations of step-down arguments are available:
\begin{enumerate}[label=(\alph*)]
\item For any ordinal $\alpha$ there is a step-down argument $\id^\alpha\vdash\alpha\leq_0\alpha$.
\item For any $\alpha<\varepsilon_0$ there is a step-down argument $s^\alpha_0\vdash0\leq_0\alpha$. If we have $\alpha\geq 1$ then there is a step-down argument $s^\alpha_1\vdash1\leq_0\alpha$. In case $\alpha\geq 2$ there is a step-down argument $s^\alpha_2\vdash2\leq_1\alpha$.
\item Assume $\alpha,\beta<\omega^2$. A step-down argument $s\vdash\alpha+1\leq_k\beta$ can be transformed into a step-down argument $3_{\cdot 2+1}^s\vdash 3^\alpha\cdot 2+1\leq_{3^{k+1}}3^\beta$ and a step down argument $3_{+1}^s\vdash 3^\alpha+1\leq_{3^{k+1}}3^\beta$. An argument $s\vdash\alpha+2\leq_k\beta$ can be transformed into an argument $3_{+2}^s\vdash3^\alpha+2\leq_{3^{k+1}}3^\beta$.
\item Assume $k\geq 2$. A step down argument $s\vdash\alpha+1\leq_k\beta$ can be transformed into a step-down argument $\omega_{\cdot 2+1}^s\vdash\omega^\alpha\cdot 2+1\leq_k\omega^\beta$ and a step-down argument $\omega_{+1}^s\vdash\omega^\alpha+1\leq_k\omega^\beta$. A step-down argument $s\vdash\alpha+2\leq_k\beta$ can be transformed into a step-down argument $\omega_{+2}^s\vdash\omega^\alpha+2\leq_k\omega^\beta$.
\end{enumerate}
We will assume that $3_{\cdot 2+1}^s$ etc.\ are step-down arguments even when the assumptions are violated (say, because the involved ordinals are bigger than $\omega^2$). For definiteness, let us agree that $3_{\cdot 2+1}^s$ denotes the step-down argument $\id^0$ in this case.
\end{lemma}
\begin{proof}
(a) In view of $\{0\}(0)=0$ we can set $\id^0:=\fund(0)^0_0$. Since any ordinal meshes with $0$ we can take $\id^\alpha:=\alpha+\id^0$ for all $\alpha>0$.\\
(b) Let us first construct $s^\alpha_1$. This is done by induction on the Cantor normal form of $\alpha$ (with repetition of the same summand rather than coefficients). Start with the case where there is only one summand. If this summand is $\omega^0$ then we have $\alpha=1$ and we can take $s^\alpha_1:=\id^1$. If the summand is of the form $\omega^\alpha$ with $\alpha>0$ we can apply the induction hypothesis to get a step-down argument $s^\alpha_1\vdash1<_0\alpha$. Then $\fund(0)*\omega^{s^\alpha_1}$ steps down from $\omega^\alpha$ via $\omega$ to $1$. If the Cantor normal form of $\alpha$ has more than one summand then we split off the last summand as $\alpha=\alpha'+\omega^{\alpha_0}$. Since $s^{\alpha'}_1$ steps down from $\alpha'$ to $1$ it suffices to step down from $\alpha$ to $\alpha'$. If we have $\alpha_0=0$ then this is achieved by $\fund(0)^\alpha$, otherwise by $\alpha'+(\fund^2(0)*\omega^{s^{\alpha_0}_1})$. Coming to $s^\alpha_0$, the case $\alpha=0$ is trivial. Otherwise one takes $s^\alpha_0=\fund(0)*s^\alpha_1$. Finally, $s^\alpha_2$ is easily constructed when $\alpha$ is a finite ordinal. Otherwise we split off the biggest summand of the Cantor normal form as $\alpha=\omega^{\alpha_0}+\alpha'$. Since $\alpha$ is infinite we must have $\alpha_0>0$. Then $\omega^{s^{\alpha_0}_1}*(\omega^{\alpha_0}+s^{\alpha'}_0)$ steps down from $\alpha$ to $\omega$. By $\fund(1)^\omega$ we reach $2$.\\
(c) Recall that any ordinal below $\omega^2$ has the form $\omega\cdot p+q$ and that we have $3^{\omega\cdot p+q}=\omega^p\cdot 3^q$ (we can even take this as our definition of exponentiation). Let us show how to construct $3^s_{\cdot 2+1}$: From $s$ we can read off $\alpha=\omega\cdot p'+q'$ and $\beta=\omega\cdot p+q$. In view of Lemma \ref{lem:step-down-properties} the argument $s$ also ensures that we have either $p'=p$ and $q'+1\leq q$ or else $p'<p$ and $q'\leq k$. In the first case $3^s_{\cdot 2+1}$ arises as the following composition of step-downs: First, one steps down from $3^\beta=\omega^p\cdot 3^q$ to $\omega^p\cdot (3^{q'}\cdot 2+1)$. In view of $3^{q'}\cdot 2+1\leq 3^{q'+1}\leq 3^q$ this can be achieved by repeated application of the argument $s^{\omega^p}_0$ from (b):
\begin{equation*}
(\omega^p\cdot (3^{q'}\cdot 2+1)+s^{\omega^p}_0)*(\omega^p\cdot (3^{q'}\cdot 2+2)+s^{\omega^p}_0)*\dots *(\omega^p\cdot (3^q-1)+s^{\omega^p}_0).
\end{equation*}
It only remains to descend from $\omega^p\cdot (3^{q'}\cdot 2+1)$ to $\omega^p\cdot 3^{q'}\cdot 2+1=3^\alpha\cdot 2+1$, which is possible by the argument $\omega^p\cdot 3^{q'}\cdot 2+s^{\omega^p}_1$. Now consider the second possibility, where we have $p'<p$ and $q'\leq k$: Similarly to the above one can step down from $3^\beta=\omega^p\cdot 3^q$ to $\omega^p$. In view of $\{\omega^{l+1}\}(0)=\omega^l$ we can then use the argument $\fund^{p-(p'+1)}(0)^{\omega^p}$ to descend from $\omega^p$ to $\omega^{p'+1}$. From there, the argument $\fund(3^{q'}\cdot 2)^{\omega^{p'+1}}$ descends to $\omega^{p'}\cdot(3^{q'}\cdot 2+1)$. Note that we have $3^{q'}\cdot 2\leq 3^{q'+1}\leq 3^{k+1}$, so that the described argument does indeed establish the relation $\leq_{3^{k+1}}$. Finally, one steps down from $\omega^{p'}\cdot(3^{q'}\cdot 2+1)$ to $\omega^{p'}\cdot3^{q'}\cdot 2+1=3^\alpha\cdot 2+1$ as in the first case. The arguments $3^s_{+1}$ and $3^s_{+2}$ are constructed similarly.\\
(d) The argument $\omega_{\cdot 2+1}^s$ arises as the following composition: Using $\omega^s$ one descends from $\omega^\beta$ to $\omega^{\alpha+1}$. From there, the argument $\fund(2)^{\omega^{\alpha+1}}$ leads to $\omega^\alpha\cdot 3$. By $\omega^\alpha\cdot 2+s^{\omega^\alpha}_1$ we reach $\omega^\alpha\cdot 2+1$. The arguments $\omega_{+1}^s$ and $\omega_{+2}^s$ are constructed similarly (for $\omega_{+2}^s$, treat $\alpha=0$ as a special case).
\end{proof}

An important insight of \cite{buchholz91} is that the rules of the infinite proof system do not have to be constructors of the terms which represent infinite proofs (in the case of the $\omega$-rule this would indeed make it impossible to stay in the realm of finite terms). Instead, these rules are read off from the term notations in a more subtle manner: For example, a finite proof term stands for the infinite proof as which it is embedded, and through this connection one can say that it ends with an $\omega$-rule. That being said, we will admit some rules as term constructors for the infinite system. This introduces redundancy but it will also make it easier to write down terms that correspond to the embedding and reduction lemma. Let us define the set $\zet^\infty_0$ of \textbf{term notations for infinite pre-proofs}:

\begin{definition}
The set $\zet^\infty_0$ consists of (untyped) terms built from the following constants and connectives:
\begin{description}
\item[($\lbrack\cdot\rbrack_M$)] A constant $[d]_M$ for each $d\in\zet_n^0$ (for some $n$) and each $M\in\mathbb N$.
\item[($\ax$)] A constant $\ax_\Gamma^\alpha$ for each $N$-axiom $\Gamma$ and each ordinal $\alpha\geq 2$.
\item[($\acc$)] A unary connective $\acc_{m,s}$ for each $m\in\mathbb N$ and each step-down $s$.
\item[($\cut$)] A binary connective $\cut_A$ for each formula $A$ in $\bigcup_{k\geq 0}\Sigma_k$ (recall that this allows $A\equiv\,n\in N$ but forbids $A\equiv\, n\notin N$).
\item[($\mathcal I$)] A unary connective $\mathcal I_{n,B}$ for each $n\in\mathbb N$ and each universal formula $B$.
\item[($\mathcal R$)] A binary connective $\mathcal R_A$ for each formula $A$ in $\bigcup_{k\geq 1}\Sigma_k$.
\item[($\mathcal E_0$)] A unary connective $\mathcal E_0$.
\item[($\mathcal E$)] A unary connective $\mathcal E$.
\end{description} 
\end{definition}

Let us briefly describe the ideas behind these constructors: The term $[d]_M$ refers to the embedding lemma. When we embed a finite proof we will have to weaken the end-sequent by some formulas of the form $n\notin N$. The index $M$ determines which of these formulas we add. The condition $\alpha\geq 2$ in the case of an axiom arises because the function $F_1$ does not grow fast enough to make the $N$-axioms true, in a sense that we will specify in the next section. The constructor $\acc_{m,s}$ refers to an accumulation rule, always with accumulation rank $i=0$, and witnessed by the step-down argument $s$. In case $m>0$ the index $m$ instructs us to weaken the end-sequent $\Gamma$ by adding the formula $(m-1)\notin N$. This will increase the number $k(\Gamma)$, which may be necessary to justify the step-down. The constructor $\cut_A$ refers to a cut. It does not come with a step-down argument because we will only need it when the ordinal is increased by precisely $1$. The constructors $\mathcal I_{n,B}$, $\mathcal R_A$ and $\mathcal E$ correspond to the proof transformations of inversion, reduction and cut-elimination, as in \cite{buchholz91}. The variant $\mathcal E_0$ is a more efficient cut eliminator for ordinals $<\omega^2$, as discussed in \cite{rathjen-carnielli-91}.\\
Elements of $\zet^\infty_0$ are called pre-proofs because not every term in $\zet^\infty_0$ will denote an infinite derivation: Recall, for example, that the cut rule may only be applied when the two premises have the same ordinal tag. To resolve this we could define $\zet^\infty$ as a set of typed terms, and allow only ``type correct" applications. This is a matter of taste, but we feel that the necessary type checking would place too much weight on the construction of terms. Instead we will later define a predicate $\zet^\infty\subseteq\zet^\infty_0$ which singles out the terms that do indeed denote proofs.\\
From any term in $\zet^\infty_0$ we need to be able to read off the \textbf{end-sequent}, the \textbf{ordinal tag}, the \textbf{cut rank}, the \textbf{accumulation rank} (i.e.\ the presence of exponential shifts in the accumulation rule), the \textbf{last rule} and the \textbf{step-down} of the proof, and terms denoting the \textbf{immediate subtrees} of the proof. Except for the last three notions, this is easy:

\begin{definition}
The end-sequent $\pend(p)$ of a pre-proof $p\in\zet^\infty_0$ is inductively defined as follows:
\begin{gather*}
\pend([d^\Gamma]_M):=\Gamma\cup\{n\notin N\, |\, n\leq M\},\\
\pend(\ax_\Gamma^\alpha):=\Gamma,\\
\pend(\acc_{m,s}p):=\begin{cases}
\pend(p) & \quad\text{if $m=0$},\\
\pend(p)\cup\{(m-1)\notin N\} & \quad\text{otherwise},
\end{cases}\\
\pend(\mathcal I_{n,B}p)=\pend(p)\backslash\{B\}\cup\{n\notin N,B_0(n)\}\quad\text{where $B\equiv\forall_xB_0(x)$},\\
\pend(\mathcal R_A p_0p_1):=\pend(\cut_A p_0p_1):=\pend(p_0)\backslash\{\neg A\}\cup\pend(p_1)\backslash\{A\},\\
\pend(\mathcal E_0p):=\pend(\mathcal E p):=\pend(p).
\end{gather*}
The ordinal tag $\pord(p)$ of a pre-proof $p\in\zet^\infty_0$ is inductively defined as follows:
\begin{gather*}
\pord([d]_M):=\omega\cdot (2\cdot\height(d)+1),\\
\pord(\ax_\Gamma^\alpha):=\alpha,\\
\pord(\acc_{m,s}p_0):=\stepto(s),\\
\pord(\cut_Ap_0p_1):=\pord(p_0)+1,\\
\pord(\mathcal I_{m,B}p):=\pord(p),\\
\pord(\mathcal E_0p):=3^{\pord(p)},\\
\pord(\mathcal Ep):=\omega^{\pord(p)},\\
\pord(\mathcal R_Ap_0p_1):=\pord(p_0)+\pord(p_1).
\end{gather*}
The cut rank $\dcut(p)$ of a pre-proof $p\in\zet^\infty_0$ is inductively defined as follows:
\begin{gather*}
\dcut([d]_M):=\dcut(d),\\
\dcut(\ax_\Gamma^\alpha):=0,\\
\dcut(\acc_{m,s}p):=\dcut(\mathcal I_{m,B}p):=\dcut(p),\\
\dcut(\cut_A p_0p_1):=\max\{k,\dcut(p_0),\dcut(p_1)\}\quad\text{where $A\in\Sigma_k$},\\
\dcut(\mathcal R_A p_0p_1):=\max\{k\dotminus 1,\dcut(p_0),\dcut(p_1)\}\quad\text{where $A\in\Sigma_k$},\\
\dcut(\mathcal E_0p):=\dcut(\mathcal E p):=\dcut(p)\dotminus 1.
\end{gather*}
The accumulation rank $\dacc(p)\in\{0,1\}$ of a pre-proof $p\in\zet^\infty_0$ is inductively defined as follows:
\begin{gather*}
\dacc([d]_M):=\dacc(\ax_\Gamma^\alpha):=0,\\
\dacc(\acc_{m,s}p):=\dacc(p),\\
\dacc(\cut_A p_0p_1):=\max\{\dacc(p_0),\dacc(p_1)\},\\
\dacc(\mathcal I_{m,A}p):=\dacc(\mathcal R_A p_0p_1):=\dacc(\mathcal E_0p):=\dacc(\mathcal Ep):=1.
\end{gather*}
\end{definition}

We can now single out those pre-proofs that denote infinite proofs:

\begin{definition}
We define a primitive recursive subset $\zet^\infty$ of $\zet^\infty_0$. Elements of $\zet^\infty$ will be called \textbf{proper proofs}:
\begin{description}
\item[($\lbrack\cdot\rbrack_M$)] A term of the form $[d]_M$ is in $\zet^\infty$ if we have $M\geq\dterm(d)$ (term depth of $d$).
\item[($\ax$)] Any term of the form $\ax_\Gamma^\alpha$ is in $\zet^\infty$.
\item[($\acc$)] If we have $p\in\zet^\infty$ and $s\vdash\pord(p)+1\leq_{{k(\pend(\acc_{m,s}p))}}\pord(\acc_{m,s}p)$ then the term $\acc_{m,s}p$ is in $\zet^\infty$.
\item[($\cut$)] If we have $p_0,p_1\in\zet^\infty$ with $\pord(p_0)=\pord(p_1)$ then $\cut_A p_0p_1$ is in $\zet^\infty$.
\item[($\mathcal I$)] If we have $p\in\zet^\infty$ then $\mathcal I_{n,B}p$ is in $\zet^\infty$.
\item[($\mathcal R$)] If $p_0,p_1$ are in $\zet^\infty$ and $\pord(p_0)$ meshes with $\pord(p_1)$ then $\mathcal R_Ap_0p_1$ is in $\zet^\infty$.
\item[($\mathcal E_0$)] If we have $p\in\zet^\infty$ with $\dacc(p)=0$ and $\pord(p)<\omega^2$ then $\mathcal E_0p$ is in $\zet^\infty$.
\item[($\mathcal E$)] If we have $p\in\zet^\infty$ then $\mathcal Ep$ is in $\zet^\infty$.
\end{description}
\end{definition}

Note that we view $\zet^\infty$ as a property, not as an inductively defined set in its own right. Thus, the principle of structural induction will only be applied to $\zet^\infty_0$, never to $\zet^\infty$. Our next goal is to exhibit primitive recursive functions which read off the last rule and the immediate subtree of a pre-proof. First, we need proof terms for the embedding of finite proofs which end with a universal introduction:

\begin{lemma}\label{lem:term-univ-embedding}
Consider a formula $A\equiv\forall_x A_0(x)$ contained in a closed sequent $\Gamma$. Assume that we have a finite proof $d_0^{\Gamma_0}\in\zet_n$ with $\Gamma_0\subseteq\Gamma\cup\{A_0(v)\}$ for some variable $v$. For all $m,M\in\mathbb N$ we will define a term $\text{Univ}(v,d_0,M,m)\in\zet^\infty_0$. If we have $M\geq\dterm(d_0)$ then the following holds:
\begin{itemize}
\item $\text{Univ}(v,d_0,M,m)\in\zet^\infty$
\item $\pord(\text{Univ}(v,d_0,M,m))=\omega\cdot 2\cdot(\height(d_0)+1)$
\item $\dacc(\text{Univ}(v,d_0,M,m))=0$
\item $\dcut(\text{Univ}(v,d_0,M,m))=\dcut(d_0)$
\item $\pend(\text{Univ}(v,d_0,M,m))\subseteq\Gamma\cup\{m\notin N,A_0(m)\}\cup\{n\notin N\, |\, n\leq M\}$
\end{itemize}
\end{lemma}
\begin{proof}
Define auxiliary pre-proofs $\text{UnivAux}(v,d_0,M,m,k)$ by induction on $k\leq m$: For $k=0$ set
\begin{equation*}
\text{UnivAux}(v,d_0,M,m,0):=[d_0[v:=m]]_{M+m}.
\end{equation*}
This is a term in $\zet^\infty_0$ because $d_0[v:=m]$ is a finite derivation with closed endsequent. For the step, define $\text{UnivAux}(v,d_0,M,m,k+1)$ as
\begin{equation*}
\cut_{(M+(m\dotminus k))\in N}(\text{UnivAux}(v,d_0,M,m,k),\ax^{\omega\cdot(2\cdot\height(d_0)+1)+k}_{\{(M+(m\dotminus (k+1)))\notin N,(M+(m\dotminus k))\in N\}}).
\end{equation*}
This is an element of $\zet^\infty_0$ because $\{(M+(m\dotminus (k+1)))\notin N,(M+(m\dotminus k))\in N\}$ is an $N$-axiom, because we have $\omega\cdot(2\cdot\height(d_0)+1)+k\geq 2$, and because the cut-formula $(M+(m\dotminus k))\in N$ is in the class $\Sigma_0\subseteq\bigcup_{k\geq 0}\Sigma_k$. Finally, set
\begin{equation}\label{def:proof-term-univ-predecessors}
\text{Univ}(v,d_0,M,m):=\acc_{m+1,\fund(m)^{\omega\cdot2\cdot(\height(d_0)+1)}}(\text{UnivAux}(v,d_0,M,m,m)).
\end{equation}
Now assume that we have $M\geq\dterm(d_0)$. It is straightforward to verify the following statements by simultaneous induction over $k\leq m$:
\begin{itemize}
\item $\text{UnivAux}(v,d_0,M,m,k)\in\zet^\infty$
\item $\pord(\text{UnivAux}(v,d_0,M,m,k))=\omega\cdot (2\cdot\height(d_0)+1)+k$
\item $\dacc(\text{UnivAux}(v,d_0,M,m,k))=0$
\item $\dcut(\text{UnivAux}(v,d_0,M,m,k))=\dcut(d_0)$
\item $\pend(\text{Univ}(v,d_0,M,m,k))\subseteq\Gamma\cup\{A_0(m)\}\cup\{n\notin N\, |\, n\leq M+(m\dotminus k)\}$
\end{itemize}
The properties demanded by the lemma are easily derived. Let us only argue for $\text{Univ}(v,d_0,M,m)\in\zet^\infty$: We need to establish
\begin{equation*}
\fund(m)^{\omega\cdot2\cdot(\height(d_0)+1)}\vdash\omega\cdot (2\cdot\height(d_0)+1)+m+1\leq_{k(\pend(\text{Univ}(v,d_0,M,m)))}\omega\cdot 2\cdot(\height(d_0)+1).
\end{equation*}
It suffices to see $m\leq k(\pend(\text{Univ}(v,d_0,M,m)))$. This holds because the formula $m\notin N$ is indeed an element of $\pend(\text{Univ}(v,d_0,M,m))$, by (\ref{def:proof-term-univ-predecessors}) and the definition of end-sequent for the constructor $\acc_{m+1,s}$.
\end{proof}

Similarly, we need proof terms for the embedding of induction rules:

\begin{lemma}\label{lem:term-induction-embedding}
Let $A[v:=m]$ be a $\Sigma_n$-formula contained in a closed sequent $\Gamma$. Let $d_0^{\Gamma_0},d_1^{\Gamma_1}\in\zet_n$ be finite proofs of the same height with $\Gamma_0\subseteq\Gamma\cup\{A[v:=0]\}$ and $\Gamma_1\subseteq\Gamma\cup\{\neg A,A[v:=Sv]\}$. For any $M\in\mathbb N$ we define a term $\ind(v,d_0,d_1,M,m)$ in $\zet^\infty_0$. If we have $M\geq\max\{\dterm(d_0),\dterm(d_1)\}$ then the following holds:
\begin{itemize}
\item $\ind(v,d_0,d_1,M,m)\in\zet^\infty$
\item $\pord(\ind(v,d_0,d_1,M,m))=\omega\cdot(2\cdot\height(d_0)+1)+m$
\item $\dacc(\ind(v,d_0,d_1,M,m))=0$
\item $\dcut(\ind(v,d_0,d_1,M,m))\leq n$
\item $\pend(\ind(v,d_0,d_1,M,m))\subseteq\Gamma\cup\{n\notin N\,|\,n\leq M\}$
\end{itemize}
\end{lemma}
\begin{proof}
We define auxiliary proofs $\text{IndAux}(v,d_1,M,k,l)$ by induction on $l\leq k$: For $l=0$ take
\begin{equation*}
\text{IndAux}(v,d_1,M,k,0):=[d_1[v:=k]]_{M+k}.
\end{equation*}
In the step, define $\text{IndAux}(v,d_1,M,k,l)$ as
\begin{equation*}
\cut_{(M+(k\dotminus l))\in N}(\text{IndAux}(v,d_1,M,k,l),\ax_{\{(M+(k\dotminus (l+1)))\notin N,(M+(k\dotminus l))\in N\}}^{\omega\cdot(2\cdot\height(d_0)+1)+l}).
\end{equation*}
Building on this, we define $\ind(v,d_0,d_1,M,k)$ by induction on $k$: For $k=0$ set
\begin{equation*}
\ind(v,d_0,d_1,M,0):=[d_0]_M.
\end{equation*}
In the step, define $\ind(v,d_0,d_1,M,k+1)$ as
\begin{equation*}
\cut_{A[v:=k]}(\text{IndAux}(v,d_1,M,k,k),\ind(v,d_0,d_1,M,k)).
\end{equation*}
It is straightforward to show the following by induction on $l\leq k$:
\begin{itemize}
\item $\text{IndAux}(v,d_1,M,k,l)\in\zet^\infty$
\item $\pord(\text{IndAux}(v,d_1,M,k,l)))=\omega\cdot(2\cdot\height(d_0)+1)+l$
\item $\dacc(\text{IndAux}(v,d_1,M,k,l))=0$
\item $\dcut(\text{IndAux}(v,d_1,M,k,l))\leq n$
\item The end-sequent of the proof $\text{IndAux}(v,d_1,M,k,l)$ is contained in the sequent $\Gamma\cup\{\neg A[v:=k],A[v:=k+1]\}\cup\{n\notin N\,|\,n\leq M+(k\dotminus l)\}$.
\end{itemize}
Building on this, it is easy to prove the following by induction on $k$:
\begin{itemize}
\item $\ind(v,d_0,d_1,M,k)\in\zet^\infty$
\item $\pord(\ind(v,d_0,d_1,M,k))=\omega\cdot(2\cdot\height(d_0)+1)+k$
\item $\dacc(\ind(v,d_0,d_1,M,k))=0$
\item $\dcut(\ind(v,d_0,d_1,M,k))\leq n$
\item $\pend(\ind(v,d_0,d_1,M,k))\subseteq\Gamma\cup\{A[v:=k]\}\cup\{n\notin N\,|\,n\leq M\}$
\end{itemize}
For the claim of the lemma, take $k=m$ and recall that $A[v:=m]$ occurs in $\Gamma$.
\end{proof}

Let us fix official tags for the rules of the infinite proof system:

\begin{definition}
The following tags are the rules of the infinite proof system:
\begin{equation*}
\ax\quad|\quad\acc^{n}\quad|\quad\lor_{j,A_0\lor A_1}\quad|\quad\land_{A_0\land A_1}\quad|\quad\cut_A\quad|\quad\exists_{m,\exists_xA}\quad|\quad\omega_{\forall_xA}
\end{equation*}
By convention we have $j\in\{0,1\}$ and $m,n\in\mathbb N$. The letter $A$ (with indices) stands for an arbitrary formula, except in the case $\cut_A$ where we demand that $A$ is in the class $\bigcup_{k\geq 0}\Sigma_k$ (recall that this allows $A\equiv\,n\in N$ but forbids $A\equiv\, n\notin N$).
\end{definition}

We remark that the rules do not carry a superscript which denotes their accumulation rank. Rather than fixing the accumulation rank of a single rule, we will work with the notion of accumulation rank for entire proofs, as defined above. One may imagine the accumulation rank of a rule as the accumulation rank of the proof in which it appears. As in \cite{buchholz91} infinite proofs should be imagined as taggings of the full $\omega$-branching tree. The superscript $n$ to the accumulation rule indicates which of the $\omega$ premises is repeated. This information will be needed for the inversion operator. We can now define primitive recursive functions which read off the last rule and the step-down of a proof, as well as terms denoting its immediate subproofs:

\begin{definition}
We define the \textbf{last rule} $\prule(p)$, the \textbf{step-down} $\step(p)$ and the \textbf{$n$-th immediate sub-proof} $\pred(p,n)$ of a pre-proof $p\in\zet^\infty_0$. First, one treats the case $p=[d]_M$ by the following case distinction:
\begin{gather*}
\prule([\ax^\Gamma]_M):=\ax,\\
\step([\ax^\Gamma]_M):=\id^0,\\
\pred([\ax^\Gamma]_M,n):=[\ax^\Gamma]_M\quad\text{for all $n\in\mathbb N$},\\[10pt]
\prule([\rep\, d_0]_M):=\acc^{0},\\
\step([\rep\, d_0]_M):=\fund^3(0)^{\omega\cdot(2\cdot\height(d_0)+3)},\\
\pred([\rep\, d_0]_M,n):=[d_0]_M\quad\text{for all $n\in\mathbb N$},\\[10pt]
\prule([d_0\land_A d_1]_M):=\land_A,\\
\step([d_0\land_A d_1]_M):=\fund^3(0)^{\omega\cdot(2\cdot\height(d_0)+3)},\\
\pred([d_0\land_A d_1]_M,n):=\begin{cases}
[d_0]_M & \quad\text{for $n\neq 1$},\\
[d_1]_M & \quad\text{for $n=1$},
\end{cases}\\[10pt]
\prule([\lor_{i,A}d_0]_M):=\lor_{i,A},\\
\step([\lor_{i,A}d_0]_M):=\fund^3(0)^{\omega\cdot(2\cdot\height(d_0)+3)},\\
\pred([\lor_{i,A}d_0]_M,n):=[d_0]_M\quad\text{for all $n\in\mathbb N$},\\[10pt]
\prule([d_0\scut_Ad_1]_M):=\cut_A,\\
\step([d_0\scut_Ad_1]_M):=\fund^3(0)^{\omega\cdot(2\cdot\height(d_0)+3)},\\
\pred([d_0\scut_Ad_1]_M):=\begin{cases}
[d_0]_M & \quad\text{for $n\neq 1$},\\
[d_1]_M & \quad\text{for $n=1$},
\end{cases}\\[10pt]
\prule([\exists_{m,A}d_0]_M):=\exists_{m,A},\\
\step([\exists_{m,A}d_0]_M):=\fund(1)*\fund^2(0)^{\omega\cdot(2\cdot\height(d_0)+3)},\\
\pred([\exists_{m,A}d_0]_M,n):=\begin{cases}
\ax^{\omega\cdot(2\cdot\height(d_0)+1)}_{m\in N,m\notin N} & \quad\text{for $n\neq 1$},\\
[d_0]_M & \quad\text{for $n=1$},
\end{cases}\\[10pt]
\prule([\forall_{v,A}d_0]_M):=\omega_A,\\
\step([\forall_{v,A}d_0]_M):=\fund(0)^{\omega\cdot(2\cdot\height(d_0)+3)},\\
\pred([\forall_{v,A}d_0]_M,n):=\text{Univ}(v,d_0,M,n),\\[10pt]
\prule([d_0\sind_{v,m,A}d_1]_M):=\acc^{0},\\
\step([d_0\sind_{v,m,A}d_1]_M):=\fund(m)*\fund^2(0)^{\omega\cdot(2\cdot\height(d_0)+3)},\\
\pred([d_0\sind_{v,m,A}d_1]_M,n):=\text{Ind}(v,d_0,d_1,M,m)\quad\text{for all $n$}.
\end{gather*}
Building on this, we define the same notions for all $p\in\zet^\infty_0$ by structural induction:
\begin{gather*}
\prule(\ax_\Gamma^\alpha):=\ax,\\
\step(\ax_\Gamma^\alpha):=\id^0,\\
\pred(\ax_\Gamma^\alpha,n):=\ax_\Gamma^\alpha\quad\text{for all $n\in\mathbb N$},\\[10pt]
\prule(\acc_{m,s}p):=\acc^{0},\\
\step(\acc_{m,s}p):=s,\\
\pred(\acc_{m,s}p,n):=p\quad\text{for all $n$},\\[10pt]
\prule(\cut_Ap_0p_1):=\cut_A,\\
\step(\cut_Ap_0p_1):=\id^{\pord(p_0)+1},\\
\pred(\cut_Ap_0p_1):=\begin{cases}
p_0 & \quad\text{for $n\neq 1$},\\
p_1 & \quad\text{for $n=1$},
\end{cases}\\[10pt]
\prule(\mathcal I_{m,B}p):=\begin{cases}
\acc^{m} & \quad\text{if $\prule(p)=\omega_{B}$},\\
\prule(p) & \quad\text{if $\prule(p)$ has a different form},
\end{cases}\\
\step(\mathcal I_{m,B}p):=\step(p),\\
\pred(\mathcal I_{m,B}p,n):=\mathcal I_{m,B}\pred(p,n),\\[10pt]
\prule(\mathcal R_Cp_0p_1):=\begin{cases}
\cut_{\neg C_0(m)} & \quad\text{if $\prule(p_1)=\exists_{m,B}$ with $B\equiv C\equiv\exists_x C_0(x)$},\\
\prule(p_1) & \quad\text{otherwise},
\end{cases}\\
\step(\mathcal R_Cp_0p_1):=\begin{cases}
\pord(p_0)+\step(p_1) & \quad\text{if $\pord(p_0)+\step(p_1)$ is a step-down argument},\\
\id^0 & \quad\text{otherwise}.
\end{cases}
\end{gather*}
The subtrees of $\mathcal R_Cp_0p_1$ are defined by a case distinction on the last rule of $p_1$. Assume first that this is not the rule $\exists_{m,C}$. Then we set
\begin{equation*}
\pred(\mathcal R_Cp_0p_1,n):=\mathcal R_Cp_0\pred(p_1,n).
\end{equation*}
In the crucial case $\prule(p_1)=\exists_{m,C}$ we set
\begin{equation*}
\pred(\mathcal R_Cp_0p_1,n):=\acc_{0,\id^{\pord(p_0)+\pord(\pred(p_1,1))+1}}(\mathcal R_C p_0\pred(p_1,1)).
\end{equation*}
for $n\neq 1$. For $n=1$, we set
\begin{equation*}
\pred(\mathcal R_Cp_0p_1,1):=\cut_{m\in N}(\acc_{0,\pord(p_0)+s^{\pord(\pred(p_1,0))}_1}\mathcal I_{m,\neg C}p_0,\mathcal R_Cp_0\pred(p_1,0))
\end{equation*}
if we have $\pord(\pred(p_1,0))>0$ and if $\pord(p_0)+s^{\pord(\pred(p_1,0))}_1$ is a step-down argument. Otherwise we set
\begin{equation*}
\pred(\mathcal R_Cp_0p_1,1):=\cut_{m\in N}(\mathcal I_{m,\neg C}p_0,\mathcal R_Cp_0\pred(p_1,0)).
\end{equation*}
Finally, we define the last rule, the step-down and the immediate subtrees of expressions that begin with a cut elimination operator:
\begin{gather*}
\prule(\mathcal E_0p):=\begin{cases}
\acc^{0} & \text{if $\prule(p)=\cut_A$ with $A\in\bigcup_{k\geq 1}\Sigma_k$},\\
\prule(p) & \text{otherwise},
\end{cases}\\
\step(\mathcal E_0p):=\begin{cases}
3^{\step(p)}_{\cdot 2+1} & \text{if $\prule(p)=\cut_A$  with $A\in\bigcup_{k\geq 1}\Sigma_k$,}\\
3^{\step(p)}_{+2} & \text{if $\prule(p)=\exists_{m,A}$,}\\
3^{\step(p)}_{+1} & \text{otherwise},
\end{cases}\\
\pred(\mathcal E_0p,n):=\begin{cases}
\mathcal R_A(\mathcal E_0\pred(p,0))(\mathcal E_0\pred(p,1)) & \text{if $\prule(p)=\cut_A$ with $A\in\bigcup_{k\geq 1}\Sigma_k$},\\
\mathcal E_0\pred(p,n) & \text{otherwise},
\end{cases}\\[10pt]
\prule(\mathcal Ep):=\begin{cases}
\acc^{0} & \text{if $\prule(p)=\cut_A$ with $A\in\bigcup_{k\geq 1}\Sigma_k$},\\
\prule(p) & \text{otherwise},
\end{cases}\\
\step(\mathcal Ep):=\begin{cases}
\omega^{\step(p)}_{\cdot 2+1} & \text{if $\prule(p)=\cut_A$ with $A\in\bigcup_{k\geq 1}\Sigma_k$,}\\
\omega^{\step(p)}_{+2} & \text{if $\prule(p)=\exists_{m,A}$,}\\
\omega^{\step(p)}_{+1} & \text{otherwise},
\end{cases}\\
\pred(\mathcal Ep,n):=\begin{cases}
\mathcal R_A(\mathcal E\pred(p,0))(\mathcal E\pred(p,1)) & \text{if $\prule(p)=\cut_A$ with $A\in\bigcup_{k\geq 1}\Sigma_k$},\\
\mathcal E\pred(p,n) & \text{otherwise}.
\end{cases}
\end{gather*}
\end{definition}

Let us show that the defined functions produce objects of the intended type:

\begin{lemma}\label{lem:rule-reads-off-rule}
The following holds for all $p\in\zet^\infty_0$ and all $n\in\mathbb N$:
\begin{enumerate}[label=(\alph*)]
\item The tag $\prule(p)$ is a rule of the infinite system, $\step(p)$ is a step-down argument, and $\pred(p,n)$ is an element of $\zet^\infty_0$.
\item If we have $\prule(p)=\ax$ then $\pred(p,n)$ and $p$ are equal.
\end{enumerate}
\end{lemma}
\begin{proof}
Both (a) and (b) are shown by a straightforward structural induction on $p$ (using Lemma \ref{lem:term-univ-embedding} and Lemma \ref{lem:term-induction-embedding} in the appropriate cases). Note that the induction step for $n_0$ relies on the induction hypothesis for $n\in\{0,1,n_0\}$.
\end{proof}

We have already defined the cut rank of a pre-proof, and we will also need to speak of the rank of a rule:

\begin{definition}
Let $\mathbf r$ be a rule of the infinite proof system. The cut rank $\dcut(\mathbf r)$ of $\mathbf r$ is defined as
\begin{equation*}
\dcut(\mathbf r):=\begin{cases}
k & \text{if $\mathbf r=\cut_A$ with $A\in\Sigma_k$},\\
0 & \text{if $\mathbf r$ has a different form}.
\end{cases}
\end{equation*}
\end{definition}

As in \cite{buchholz91} the conditions on the infinite rules are formulated as \textbf{local correctness conditions}:

\begin{definition}
A pre-proof $p\in\zet^\infty_0$ is locally correct if we have $\lc(p,n)$ for all $n\in\mathbb N$, where $\lc(p,n)$ denotes the conjunction of the following conditions:
\begin{description}
\item[$\lbrack\lc_{\text{cut}}(p,n)\rbrack$] We have $\dcut(p)\geq\dcut(\prule(p))$ and $\dcut(p)\geq\dcut(\pred(p,n))$.
\item[$\lbrack\lc_{\text{acc}}(p,n)\rbrack$] We have $\dacc(p)\geq\dacc(\pred(p,n))$.
\item[$\lbrack\lc_{\text{step}}(p,n)\rbrack$] If $\prule(p)=\ax$ then we have $\pord(p)\geq 2$. If $\prule(p)$ is of the form $\exists_{m,A}$ then we have
\begin{equation*}
\step(p)\vdash\pord(\pred(p,n))+2\leq_{3_{\dacc(p)+}^{k(\pend(p))}}\pord(p).
\end{equation*}
If $\prule(p)$ is of a different form then we have
\begin{equation*}
\step(p)\vdash\pord(\pred(p,n))+1\leq_{3_{\dacc(p)+}^{k(\pend(p))}}\pord(p).
\end{equation*}
\item[$\lbrack\lc_{\text{end}}(p,n)\rbrack$] The conditions on the end-sequent depend on the last rule of $p$:
\begin{itemize}
\item If $\prule(p)=\ax$ then the sequent $\pend(p)$ contains an ($N$- or \mbox{truth-)} axiom.
\item If $\prule(p)=\acc^{m}$ with $m=n$ then $\pend(\pred(p,m))\subseteq\pend(p)$ holds.
\item If $\prule(p)=\lor_{j,A_0\lor A_1}$ then the formula $A_0\lor A_1$ occurs in $\pend(p)$ and we have $\pend(\pred(p,0))\subseteq\pend(p)\cup\{A_j\}$.
\item If $\prule(p)=\land_{A_0\land A_1}$ then the formula $A_0\land A_1$ occurs in $\pend(p)$, and $\pend(\pred(p,0))\subseteq\pend(p)\cup\{A_0\}$ and $\pend(\pred(p,1))\subseteq\pend(p)\cup\{A_1\}$ hold.
\item If $\prule(p)=\cut_A$ then we have $\pend(\pred(p,0))\subseteq\pend(p)\cup\{\neg A\}$ and $\pend(\pred(p,1))\subseteq\pend(p)\cup\{A\}$.
\item If $\prule(p)=\exists_{m,\exists_x A(x)}$ then the formula $\exists_x A(x)$ occurs in $\pend(p)$ and we have $\pend(\pred(p,0))\subseteq\pend(p)\cup\{m\in N\}$ and $\pend(\pred(p,1))\subseteq\pend(p)\cup\{A(m)\}$.
\item If $\prule(p)=\omega_{\forall_x A(x)}$ then the formula $\forall_x A(x)$ occurs in $\pend(p)$ and we have $\pend(\pred(p,n))\subseteq\pend(p)\cup\{n\notin N,A(n)\}$.
\end{itemize}
\end{description}
\end{definition}

The following observation will be used frequently:

\begin{lemma}\label{lem:pred-ordinal-independent-n}
 Consider $p\in\zet^\infty_0$. If the conditions $\lc(p,n)$ and $\lc(p,m)$ hold then we have $\pord(\pred(p,n))=\pord(\pred(p,m))$. In particular, $\pord(\pred(p,n))$ is independent of $n$ if $p$ is locally correct.
\end{lemma}
\begin{proof}
 In case $\prule(p)=\ax$ the claim holds by Lemma \ref{lem:rule-reads-off-rule}. Otherwise, the condition $\lc_{\text{step}}(p,n)$ allows to read off $\pord(\pred(p,n))$ from $\step(p)$, in a way that is indeed independent of $n$.
\end{proof}

The following result is parallel to \cite[Theorem 3.8]{buchholz91}. It shows that the elements of $\zet^\infty$ behave, in the relevant respects, like actual infinite proofs. Particular to our set-up, we must also show that the subset $\zet^\infty\subset\zet^\infty_0$, i.e.\ the set of proper proofs, is closed under the relevant operations.

\begin{proposition}[$\isigma_1$]\label{prop:pred-proper-proof-proper}
For any proper proof $p\in\zet^\infty$ and any $n\in\mathbb N$ the following holds:
\begin{enumerate}[label=(\alph*)]
 \item We have $\pred(p,n)\in\zet^\infty$, i.e.\ $\pred(p,n)$ is a proper proof.
 \item We have $\lc(p,n)$.
\end{enumerate}
Thus $p$ is locally correct.
\end{proposition}
\begin{proof}
We show (a) and (b) by simultaneous structural induction on $p\in\zet^\infty_0$. The induction step for $n_0\in\mathbb N$ requires the induction hypothesis for $n\in\{0,1,n_0\}$.\\
First, one verifies the proposition for proofs of the form $p=[d^\Gamma]_M$, by a case distinction on the last rule of the finite proof $d$. We only elaborate a few cases: Assume first that $d^\Gamma$ is of the form $\exists_{m,A}d_0^{\Gamma_0}$ with $A\equiv\exists_x A_0(x)$, $A\in\Gamma$ and $\Gamma_0\subseteq\Gamma\cup\{A_0(m)\}$.
Concerning (a), the case $n\neq 1$ is immediate since we have $\ax^{\omega\cdot(2\cdot\height(d_0)+1)}_{m\in N,m\notin N}\in\zet^\infty$ without any further conditions. In case $n=1$ we argue as follows:
\begin{equation*}
[d]_M\in\zet^\infty\quad\Rightarrow\quad M\geq\dterm(d)\quad\Rightarrow\quad M\geq\dterm(d_0)\quad\Rightarrow\quad[d_0]_M\in\zet^\infty
\end{equation*}
Coming to (b), the condition $\lc_{\text{cut}}(p,n)$ holds because we have
\begin{equation*}
\dcut(\exists_{m,A}^0)=0=\dcut(\ax^{\omega\cdot(2\cdot\height(d_0)+1)}_{m\in N,m\notin N})
\end{equation*}
and
\begin{equation*}
\dcut([d_0]_M)=\dcut(d_0)=\dcut(d)=\dcut([d]_M).
\end{equation*}
The condition $\lc_{\text{acc}}(p,n)$ follows from
\begin{equation*}
\dacc(\ax^{\omega\cdot(2\cdot\height(d_0)+1)}_{m\in N,m\notin N})=\dacc([d_0]_M)=0.
\end{equation*}
As for the condition $\lc_{\text{step}}(p,n)$, observe that $\pord(\pred([d]_M,n))=\omega\cdot(2\cdot\height(d_0)+1)$ holds independently of $n$. Because of $\height(d)=\height(d_0)+1$ we have $\pord([d]_M)=\omega\cdot(2\cdot\height(d_0)+3)$. Now in view of $k(\pend([d]_M))\geq 1$ it suffices to establish
\begin{equation*}
\fund(1)*\fund^2(0)^{\omega\cdot(2\cdot\height(d_0)+3)}\vdash\omega\cdot(2\cdot\height(d_0)+1)+2\leq_1 \omega\cdot(2\cdot\height(d_0)+3).
\end{equation*}
This is indeed true, since we have
\begin{gather*}
 \{\omega\cdot (2\cdot\height(d_0)+3)\}(0)=\omega\cdot (2\cdot\height(d_0)+2)+1,\\
 \{\omega\cdot (2\cdot\height(d_0)+2)+1\}(0)=\omega\cdot (2\cdot\height(d_0)+2),\\
 \{\omega\cdot (2\cdot\height(d_0)+2)\}(1)=\omega\cdot (2\cdot\height(d_0)+1)+2.
\end{gather*}
Finally, we verify the condition $\lc_{\text{end}}(p,n)$: Observe first that we have
\begin{equation*}
A\in\Gamma\subseteq\Gamma\cup\{n\notin N\, |\, n\leq M\}=\pend([d]_M)
\end{equation*}
and
\begin{multline*}
\pend(\pred([d]_M,1))=\pend([d_0]_M)=\Gamma_0\cup\{n\notin N\, |\, n\leq M\}\subseteq\\
\subseteq\Gamma\cup\{A_0(m)\}\cup\{n\notin N\, |\, n\leq M\}=\pend([d]_M)\cup\{A_0(m)\}.
\end{multline*}
In view of $M\geq\dterm(d)\geq m$ the formula $m\notin N$ occurs in the sequent $\pend([d]_M)$. Thus we also have
\begin{multline*}
\pend(\pred([d]_M,0))=\pend(\ax^{\omega\cdot(2\cdot\height(d_0)+1)}_{m\in N,m\notin N})=\\
=\{m\in N,m\notin M\}\subseteq\pend([d]_M)\cup\{m\in N\},
\end{multline*} 
just as required.\\
Let us also consider the case $p=[d^\Gamma]_M$ with $d^\Gamma=d_0^{\Gamma_0}\sind_{v,m,A}d_1^{\Gamma_1}$. The assumption $p\in\zet^\infty$ tells us $M\geq\dterm(d)=\max\{m,\dterm(d_0),\dterm(d_1)\}$. The conditions on finite proofs provide the assumptions of Lemma \ref{lem:term-induction-embedding}. From the lemma we first learn
\begin{equation*}
\pend(p,n)=\ind(v,d_0,d_1,M,m)\in\zet^\infty,
\end{equation*}
as required for (a). Coming to (b), by definition the finite derivation $d$ has cut rank $n$, because it involves an induction inference. Then Lemma \ref{lem:term-induction-embedding} ensures the condition $\lc_{\text{cut}}(p,n)$. The condition $\lc_{\text{acc}}(p,n)$ also follows from Lemma \ref{lem:term-induction-embedding}. Coming to the condition $\lc_{\text{step}}(p,n)$, note that Lemma \ref{lem:term-induction-embedding} implies
\begin{equation*}
\pord(\pred(p,n))=\pord(\ind(v,d_0,d_1,M,m))=\omega\cdot(2\cdot\height(d_0)+1)+m.
\end{equation*}
Clearly we have
\begin{multline*}
\fund(m)*\fund^2(0)^{\omega\cdot(2\cdot\height(d_0)+3)}\vdash\\
\omega\cdot(2\cdot\height(d_0)+1)+m+1\leq_m \omega\cdot(2\cdot\height(d_0)+3).
\end{multline*}
All that remains is to check $m\leq k(\pend(p))$. Indeed, we have already observed the inequality $M\geq m$, which implies that the formula $m\notin N$ does occur in the sequent $\pend(p)=\Gamma\cup\{n\notin N\, |\, n\leq M\}$. To verify the condition $\lc_{\text{end}}(p,n)$ we only need to invoke Lemma \ref{lem:term-induction-embedding}:
\begin{equation*}
\pend(\pred(p,0))=\pend(\ind(v,d_0,d_1,M,m))\subseteq\Gamma\cup\{n\notin N\, |\, n\leq M\}=\pend(p)
\end{equation*}
The case $p=[d^\Gamma]_M$ with $d^\Gamma=\forall_{v,A}d_0^{\Gamma_0}$ is treated similarly, now with the help of Lemma \ref{lem:term-univ-embedding}.\\
So far we have verified the proposition for proofs $p$ of the form $[d]_M$. The cases $p=\ax_\Gamma^\alpha$, $p=\acc_{m,s}p_0$ and $p=\cut_Ap_0p_1$ are easily checked by unravelling the definitions (no applications of the induction hypothesis).\\
Now let us consider the case $p=\mathcal I_{m,B}p_0$ with $B\equiv\forall_x B_0(x)$. From the assumption $p\in\zet^\infty$ we infer $p_0\in\zet^\infty$. Thus (a) follows immediately from the induction hypothesis. Coming to (b), the condition $\lc_{\text{cut}}(p,n)$ carries over from the induction hypothesis since we have $\dcut(p)=\dcut(p_0)$, $\dcut(\prule(p))=\dcut(\prule(p_0))$ and
\begin{equation*}
\dcut(\pred(p,n))=\dcut(\mathcal I_{m,B}\pred(p_0,n))=\dcut(\pred(p_0,n)).
\end{equation*}
The condition $\lc_{\text{acc}}(p,n)$ holds because $\dacc(p)=1$ is already the maximal possible value. To verify the condition $\lc_{\text{step}}(p,n)$, we distinguish three cases according to the form of $\prule(p_0)$: If we have $\prule(p_0)=\ax$ then we also have $\prule(p)=\ax$ and there is nothing to check. If we have $\prule(p_0)=\exists_{m,A}$ then the induction hypothesis tells us
\begin{equation*}
\step(p_0)\vdash\pord(\pred(p_0,n))+2\leq_{3_{\dacc(p_0)+}^{\pend(p_0)}}\pord(p_0).
\end{equation*}
Since we have $\step(p)=\step(p_0)$, $\pord(p)=\pord(p_0)$ and $\pord(\pred(p,n))=\pord(\mathcal I_{m,B}\pred(p_0,n))=\pord(\pred(p_0,n))$ we may conclude
\begin{equation*}
\step(p)\vdash\pord(\pred(p,n))+2\leq_{3_{\dacc(p_0)+}^{k(\pend(p_0))}}\pord(p).
\end{equation*}
In view of $\prule(p)=\exists_{m,A}$ it only remains to establish the inequality
\begin{equation*}
{3_{\dacc(p_0)+}^{k(\pend(p_0))}}\leq {3_{\dacc(p)+}^{k(\pend(p))}}.
\end{equation*}
First, we have $\dacc(p_0)\leq 1=\dacc(p)$. Furthermore, $k(\pend(p_0))\leq k(\pend(p))$ holds because we have
\begin{equation}\label{eq:inclusion-end-sequents-inversion}
\pend(p_0)\subseteq\pend(p)\cup\{B\}
\end{equation}
and because the universal formula $B$ cannot be of the form $n\notin N$. The case where $\prule(p_0)$ is neither of the form $\ax$ nor of the form $\exists_{m,A}$ is similar. The condition $\lc_{\text{end}}(p,n)$ is checked by a case distinction on $\prule(p_0)$: If we have $\prule(p_0)=\ax=\prule(p)$ then the induction hypothesis tells us that $\pend(p_0)$ contains an axiom. Since all formulas which appear in axioms are atomic it follows from (\ref{eq:inclusion-end-sequents-inversion}) that the same axiom is still contained in $\pend(p)$. Most other cases follow similarly, except for $\prule(p_0)=\omega_B$. There we have $\prule(p)=\acc^m$, and $\lc_{\text{end}}(p,n)$ is vacuously true except if $m=n$. In the latter case we use the induction hypothesis to show that we have
\begin{multline*}
 \pend(\pred(p,n))=\pend(\mathcal I_{n,B}\pred(p_0,n))=\\
=\pend(\pred(p_0,n))\backslash\{B\}\cup\{n\notin N,B_0(n)\}\subseteq\\
\subseteq(\pend(p_0)\cup\{n\notin N,B_0(n)\})\backslash\{B\}\cup\{n\notin N,B_0(n)\}\subseteq\\
\subseteq\pend(p_0)\backslash\{B\}\cup\{n\notin N,B_0(n)\}=\pend(p),
\end{multline*}
as required.\\
Let us turn to the case $p=\mathcal R_Cp_0p_1$ with $C\in\Sigma_k$ for some $k\geq 1$. The assumption $p\in\zet^\infty$ tells us that $\pord(p_0)$ meshes with $\pord(p_1)$, and that we have $p_0,p_1\in\zet^\infty$. Concerning all statements that we need to verify, let us distinguish two cases: Assume first that $\prule(p_1)$ is not of the form $\exists_{m,B}$ with $B\equiv C$. Concerning (a), the induction hypothesis tells us that $\pred(p_1,n)$ is an element of $\zet^\infty$. To derive that $\pred(p,n)=\mathcal R_Cp_0\pred(p_1,n)$ is an element of $\zet^\infty$ we still need to see that $\pord(p_0)$ meshes with $\pord(\pred(p_1,n))$. This follows from $\pord(\pred(p_1,n))\leq\pord(p_1)$, which holds by $\lc_{\text{step}}(p_1,n)$ if $\prule(p_1)\neq\ax$ or else by Lemma \ref{lem:rule-reads-off-rule}. The condition $\lc_{\text{cut}}(p_1,n)$ follows from the induction hypothesis by
\begin{equation*}
 \dcut(p)\geq\dcut(p_1)\geq\dcut(\prule(p_1))=\dcut(\prule(p))
\end{equation*}
and
\begin{multline*}
\dcut(p)=\max\{k\dotminus 1,\dcut(p_0),\dcut(p_1)\}\geq\\
\geq\max\{k\dotminus 1,\dcut(p_0),\dcut(\pred(p_1,n))\}=\dcut(\pred(p,n)).
\end{multline*}
The condition $\lc_{\text{acc}}(p,n)$ holds because $\dacc(p)=1$ is the maximal possible value. The condition $\lc_{\text{step}}(p,n)$ is checked by a case distinction on $\prule(p_1)$. We only treat the case $\prule(p_1)=\exists_{m,A}=\prule(p)$ (with $A\not\equiv C$) explicitly, all other cases being similar. The induction hypothesis gives
\begin{equation*}
 \step(p_1)\vdash\pord(\pred(p_1,n))+2\leq_{3_{\dacc(p_1)+}^{k(\pend(p_1))}}\pord(p_1).
\end{equation*}
Since $\pord(p_0)$ meshes with $\pord(p_1)$ we know that $\pord(p_0)+\step(p_1)$ is a step-down argument. This is thus the argument $\step(p)$, and by $\pord(p)=\pord(p_0)+\pord(p_1)$ and $\pord(\pred(p,n))=\pord(p_0)+\pord(\pred(p_1,n))$ we can conclude
\begin{equation*}
 \step(p)\vdash\pord(\pred(p,n))+2\leq_{3_{1+}^{k(\pend(p_1))}}\pord(p).
\end{equation*}
All that remains is to check $k(\pend(p_1))\leq k(\pend(p))$. This follows from
\begin{equation*}
 \pend(p_1)\subseteq\pend(p)\cup\{C\},
\end{equation*}
since the formula $C\in\Sigma_k$ with $k\geq 1$ cannot be of the form $n\notin N$. The condition $\lc_{\text{end}}(p,n)$ is checked by a straightforward case distinction on $\prule(p)=\prule(p_1)$.\\
Still concerning $p=\mathcal R_Cp_0p_1$, let us turn to the crucial case $\prule(p_1)=\exists_{m,C}$. As before, the induction hypothesis implies that $\pord(p_0)$ meshes with $\pord(\pred(p_1,0))$. It also tells us that we have $\pord(\pred(p_1,0))=\pord(\pred(p_1,1))$: Indeed, the condition $\lc_{\text{end}}(p_1,n)$ yields $\pord(\pred(p_1,n))+2=\stepbo(\step(p_1))$, independently of $n$. Concerning (a), it is now straightforward to check $\pred(p,n)\in\zet^\infty$ in all possible cases. Coming to (b), let us first consider the condition $\lc_{\text{cut}}(p,n)$: From $C\in\Sigma_k$ we can infer $\neg C_0(m)\in\Sigma_{k\dotminus 1}$, where we have $C\equiv\exists_x C_0(x)$. Thus we have
\begin{multline*}
 \dcut(p)=\max\{k\dotminus 1,\dcut(p_0),\dcut(p_1)\}\geq k\dotminus 1=\dcut(\cut_{\neg C_0(m)})=\dcut(\prule(p)).
\end{multline*}
Also, it is straightforward to verify
\begin{equation*}
 \dcut(\pred(p,n))\leq\max\{k\dotminus 1,\dcut(p_0),\dcut(\pred(p_1,0)),\dcut(\pred(p_1,1))\}
\end{equation*}
in all possible cases. Using the induction hypothesis we can conclude
\begin{equation*}
 \dcut(p)=\max\{k\dotminus 1,\dcut(p_0),\dcut(p_1)\}\geq\dcut(\pred(p,n)).
\end{equation*}
The condition $\lc_{\text{acc}}(p,n)$ is trivial, as above. Coming to the condition $\lc_{\text{step}}(p,n)$, the induction hypothesis yields
\begin{equation*}
 \step(p_1)\vdash\pord(\pred(p_1,0))+2\leq_{3_{\dacc(p_1)+}^{k(\pend(p_1))}}\pord(p_1).
\end{equation*}
It is straightforward to check that
\begin{equation*}
 \pord(\pred(p,n))=\pord(p_0)+\pord(\pred(p_1,0))+1
\end{equation*}
holds in all possible cases, as well as $\pord(p)=\pord(p_0)+\pord(p_1)$. As above we have $\step(p)=\pord(p_0)+\step(p_1)$ and $k(\pend(p_1))\leq k(\pend(p))$. We can conclude
\begin{equation*}
 \step(p)\vdash\pord(\pred(p,n))+1\leq_{3_{1+}^{k(\pend(p))}}\pord(p),
\end{equation*}
as required. Finally, the condition $\lc_{\text{acc}}(p,n)$ follows from the induction hypothesis as
\begin{multline*}
 \pend(\pred(p,0))=\pend(\mathcal R_C p_0\pred(p_1,1))=\\
=\pend(p_0)\backslash\{\neg C\}\cup\pend(\pred(p_1,1))\backslash\{C\}\subseteq\\
\subseteq\pend(p_0)\backslash\{\neg C\}\cup(\pend(p_1)\cup\{C_0(m)\})\backslash\{C\}\subseteq\pend(p)\cup\{\neg\neg C_0(m)\}
\end{multline*}
and
\begin{multline*}
 \pend(\pred(p,1))=\pend(\mathcal I_{m,\neg C} p_0)\backslash\{m\notin N\}\cup\pend(\mathcal R_C p_0\pred(p_1,0))\backslash\{m\in N\}\subseteq\\
\subseteq(\pend(p_0)\backslash\{\neg C\}\cup\{\neg C_0(m)\})\cup\pend(\pred(p_1,0))\backslash\{C,m\in N\}\subseteq\\
\subseteq\pend(p_0)\backslash\{\neg C\}\cup\{\neg C_0(m)\}\cup(\pend(p_1)\cup\{m\in M\})\backslash\{C,m\in N\}\subseteq\\
\subseteq\pend(p_0)\backslash\{\neg C\}\cup\{\neg C_0(m)\}\cup\pend(p_1)\backslash\{C\}=\pend(p)\cup\{\neg C_0(m)\},
\end{multline*}
as required.\\
We come to the case $p=\mathcal E_0p_0$. The assumption $p\in\zet^\infty$ tells us $\pord(p_0)<\omega^2$ and $\dacc(p_0)=0$, as well as $p_0\in\zet^\infty$. Concerning all statements we need to verify, assume first that $\prule(p_0)$ is not of the form $\cut_A$ with $A\in\bigcup_{k\geq 1}\Sigma_k$. It is easy to deduce (a) from the induction hypothesis. Note in particular that we have $\pord(\pred(p_0,n))\leq\pord(p_0)<\omega^2$ by the condition $\lc_{\text{step}}(p_0,n)$ and by Lemma \ref{lem:rule-reads-off-rule}, as already observed above. Also, the condition $\lc_{\text{acc}}(p_0,n)$ ensures $\dacc(\pred(p_0,n))\leq\dacc(p_0)=0$. Coming to (b), the condition $\lc_{\text{cut}}(p,n)$ is easily deduced from the induction hypothesis (note that $\dcut(\prule(p))=0$ holds by the assumption of the case distinction), and the condition $\lc_{\text{acc}}(p,n)$ follows from $\dacc(p)=1$. As for the condition $\lc_{\text{step}}(p,n)$, we only comment on the case $\prule(p_0)=\exists_{m,B}=\prule(p)$ explicitly, the other cases being similar: There, the induction hypothesis tells us
\begin{equation*}
 \step(p_0)\vdash\pord(\pred(p_0,n))+2\leq_{k(\pend(p_0))}\pord(p_0).
\end{equation*}
Note $\pord(\pred(p,n))=\pord(\mathcal E_0\pred(p_0,n))=3^{\pord(\pred(p_0,n))}$, $\pord(p)=3^{\pord(p_0)}$ and $\step(p)=3^{\step(p_0)}_{+2}$.
Then it suffices to invoke Lemma \ref{lem:transformations-step-downs} to get the required
\begin{equation*}
 \step(p)\vdash\pord(\pred(p,n))+2\leq_{3^{k(\pend(p))}_{1+}}\pord(p).
\end{equation*}
The condition $\lc_{\text{end}}(p,n)$ carries over from the induction hypothesis, because we have $\prule(p)=\prule(p_0)$, as well as  $\pend(p)=\pend(p_0)$ and $\pend(\pred(p,n))=\pend(\mathcal E_0\pred(p_0,n))=\pred(\pred(p_0,n))$.\\
Still for $p=\mathcal E_0p_0$, let us come to the case $\prule(p_0)=\cut_A$ with $A\in\bigcup_{k\geq 1}\Sigma_k$. We begin with (a): As above, the terms $\mathcal E_0\pred(p_0,0)$ and $\mathcal E_0\pred(p_0,1)$ are in $\zet^\infty$. To conclude that $\pred(p,n)=\mathcal R_A(\mathcal E_0\pred(p_0,0))(\mathcal E_0\pred(p_0,1))$ is in $\zet^\infty$ we still need to know that the ordinal $\pord(\mathcal E_0\pred(p_0,0))=3^{\pord(\pred(p_0,0))}$ meshes with $\pord(\mathcal E_0\pred(p_0,1))=3^{\pord(\pred(p_0,1))}$. This holds because $\pord(\pred(p_0,0))$ and $\pord(\pred(p_0,1))$ are the same ordinal, and since $3^\alpha$ meshes with $3^\alpha$ for any $\alpha<\omega^2$. Coming to (b), let us first discuss the condition $\lc_{\text{cut}}(p,n)$: First, observe that $\prule(p)=\acc^0$ means $\dcut(\prule(p))=0$. Next, by induction hypothesis the assumption $\prule(p_0)=\cut_A$ implies $\dcut(p_0)\geq k$. Again using the induction hypothesis we thus obtain the required
\begin{multline*}
 \dcut(\pred(p,n))=\max\{k\dotminus 1,\dcut(\mathcal E_0\pred(p_0,0)),\dcut(\mathcal E_0\pred(p_0,1)))\}=\\
=\max\{k\dotminus 1,\dcut(\pred(p_0,0))\dotminus 1,\dcut(\pred(p_0,1))\dotminus 1\}\leq\dcut(p_0)\dotminus 1=\dcut(p).
\end{multline*}
The condition $\lc_{\text{acc}}(p,n)$ holds because of $\dacc(p)=1$. As for $\lc_{\text{step}}(p,n)$, the induction hypothesis yields
\begin{equation*}
 \step(p_0)\vdash\pord(\pred(p_0,n))+1\leq_{k(\pend(p_0))}\pord(p_0).
\end{equation*}
We also have
\begin{equation*}
 \pord(\pred(p,n))=3^{\pord(\pred(p_0,0))}+3^{\pord(\pred(p_0,1))}=3^{\pord(\pred(p_0,n))}\cdot 2.
\end{equation*}
Using Lemma \ref{lem:transformations-step-downs} we thus get
\begin{equation*}
 \step(p)\vdash\pord(\pred(p,n))+1\leq_{3_{1+}^{k(\pend(p))}}\pord(p).
\end{equation*}
just as required. Finally, the condition $\lc_{\text{end}}(p,n)$ holds because of
\begin{multline*}
 \pend(\pred(p,0))=\pend(\pred(p_0,0))\backslash\{\neg A\}\cup\pend(\pred(p_0,1))\backslash\{A\}\subseteq\\
\subseteq(\pend(p_0)\cup\{\neg A\})\backslash\{\neg A\}\cup(\pend(p_0)\cup\{A\})\backslash\{A\}\subseteq\pend(p_0)=\pend(p).
\end{multline*}
The case $p=\mathcal E p_0$ is similar to $p=\mathcal E_0 p_0$. To apply Lemma \ref{lem:transformations-step-downs}(d), note that $3_{1+}^k\geq 3$ holds for all $k$. Also observe that $\omega^\alpha$ meshes with $\omega^\alpha$, independently of $\alpha$.
\end{proof}

\section{Reading off Bounds from Infinite Proofs}

In the last section we have introduced finite terms which model transformations of infinite proofs, in particular cut elimination. Here, we show how a witness for an existential statement can be read off from an infinite proof of sufficiently low cut rank. In particular, the existence of such a witness yields the reflection principle.\\
More precisely, we will show that a closed $\Sigma_1$-formula $\varphi$ has a witness $n$ with $3^{n+1}<F_\alpha(3^{\max\{1,M\}+1})$ if we have a proof $p\in\zet^\infty$ with $\pord(p)=\alpha$, $\dcut(p)\leq 1$ and
\begin{equation*}
 \pend(p)\subseteq\{\varphi\}\cup\{n\notin N\,|\, n\leq M\}.
\end{equation*}
To establish this claim we will need to manipulate the fast-growing functions which provide the existential bounds. In the theory $\isigma_1$ (even in $\pa$) this is only possible if we provide some bound in advance. Then we can work with the following function (recall from \cite[Section 5.2]{sommer95} that the relation $F_\alpha(n)=m$ is $\Delta_0$-definable, and thus primitive recursive):

\begin{definition}
 We define a primitive recursive function $\pord\times\mathbb N\times\mathbb N\rightarrow\mathbb N$ by setting
\begin{equation*}
 F_\alpha(n;K):=\begin{cases}
                 F_\alpha(n) & \text{if we have $F_\alpha(n)\leq K$},\\
                 K & \text{otherwise}.
                \end{cases}
\end{equation*}
\end{definition}
  
To define the axioms of the infinite proof system, we have already made use of a primitive recursive truth definition for proper $\Delta_0$-formulas (i.e.\ excluding formulas of the form $n\in N$). Now we extend the truth definition to the whole class $\Delta_0\cup\Sigma_1$. Recall that $\Delta_0$ contains the formulas $n\in N$ but not the formulas $n\notin N$, and that none of the special formulas may appear as a proper subformula of a compound formula.

\begin{definition}\label{def:truth-finite-universe}
Given a number $K\in\mathbb N$ we set
\begin{equation*}
 \true(n\in N;K):=\text{``we have $3^{n+1}<K$''}.
\end{equation*}
If $\psi$ is a closed proper $\Delta_0$-formula then we set
\begin{equation*}
 \true(\psi;K):=\text{``the formula $\psi$ is true''}.
\end{equation*}
For a closed $\Sigma_1$-formula $\varphi\equiv\exists_x\psi(x)$ we set
\begin{equation*}
 \true(\varphi;K):=\exists_{n\leq K}(\true(n\in N;K)\land\true(\psi(n);K)).
\end{equation*}
Note that the bound $n\leq K$ is redundant by the first conjunct.\\
Given a sequent $\Gamma$ we set
\begin{equation*}
 \true(\Gamma;K):=\text{``we have $\true(\varphi;K)$ for some formula $\varphi\in\Gamma\cap(\Delta_0\cup\Sigma_1)$''.}
\end{equation*}
We will only use this notion if all formula in $\Gamma$ are in the class $\Delta_0\cup\Sigma_1$ or of the form $n\notin N$. Such sequents will be called $\Sigma_1$-sequents. We will also use the abbreviation
\begin{equation*}
 \false(\Gamma;K):=\neg\true(\Gamma;K),
\end{equation*}
and similarly for single formulas.
\end{definition}

We need to see that this truth definition is monotone:

\begin{lemma}\label{lem:truth-bound-monotonous}
Consider sequents $\Gamma$, $\Gamma'$ with $\Gamma\cap(\Delta_0\cup\Sigma_1)\subseteq\Gamma'$, and bounds $K\leq K'$. Then $\true(\Gamma;K)$ implies $\true(\Gamma';K')$.
\end{lemma}
\begin{proof}
Monotonicity in the bound is easily esyablished for formulas of the form $n\in N$, for closed proper $\Delta_0$-formulas, and then for closed $\Sigma_1$-formulas. The statement for sequents follows.
\end{proof}

Recall the situation above, namely that we have a proof $p\in\zet^\infty$ with $\pord(p)=\alpha$, $\dcut(p)\leq 1$ and
\begin{equation*}
 \pend(p)\subseteq\{\varphi\}\cup\{n\notin N\,|\, n\leq M\}
\end{equation*}
for some $\Sigma_1$-formula $\varphi$. Using the new terminology, our goal is to show that we have $\true(\pend(p);K)$ with $K=F_{\pord(p)}(3^{k(\pend(p))}_{1+})$. Under the contradictory assumption that this is not the case, the pair $(p,K)$ is bad in the following sense:

\begin{definition}
 A pair $(p,K)\in\zet^\infty_0\times\mathbb N$ is called bad if the following holds:
\begin{itemize}
 \item $p\in\zet^\infty$, i.e.\ $p$ is a proper proof,
 \item $\dcut(p)\leq 1$,
 \item $\pend(p)$ is a $\Sigma_1$-sequent, i.e.\ contains only formulas from $\Delta_0\cup\Sigma_1$ and formulas of the form $n\notin N$,
 \item $K=F_{\pord(p)}(3^{k(\pend(p))}_{1+})$,
 \item $\false(\pend(p);K)$.
\end{itemize}
\end{definition}

Our goal is to show that there cannot be a bad pair. To achieve this, we will assume that there is a bad pair $(p,K)$. Starting with $(p_0,K_0)=(p,K)$ we will then construct a sequence of bad pairs $(p_n,K_n)$ with the additional absurd property that $K_{n+1}<K_n$ holds for all $n$. Loosely speaking, if $(p_n,K_n)$ is a bad pair then the end-sequent of $p_n$ is false. By the local correctness of $p_n\in\zet^\infty$ this implies that some immediate subproof provided a false premise. In most cases we can simply take $p_{n+1}$ to be that subproof. If the last rule was a cut over a $\Sigma_1$-formula then we may additionally have to invert on the immediate subproof, to assure that its end-sequent is still in $\Sigma_1$. Let us define the functions which construct the described sequence:

\begin{definition}
We define primitive recursive functions $\stepproof:\zet^\infty_0\times\mathbb N\rightarrow\zet^\infty_0$ and $\stepbound:\zet^\infty_0\times\mathbb N\rightarrow\mathbb N$ by the following case distinction:
\begin{gather*}
\stepproof(p,K):=\begin{cases}
                  \pred(p,n) & \text{if $\prule(p)=\acc^n$},\\[4pt]
                  \pred(p,0) & \parbox[t]{8cm}{\raggedright if $\prule(p)=\cut_A$ with $A\in\Delta_0$ and if $\true(A;F_{\pord(\pred(p,1))}(3^{k(\pend(\pred(p,1)))}_{1+};K))$,}\\[20pt]
                  \mathcal I_{m,\neg A}\pred(p,0) & \parbox[t]{8cm}{\raggedright if $\prule(p)=\cut_A$ with $A\equiv\exists_x A_0(x)\in\Sigma_1$ and if $m\leq K$ is minimal such that both $\true(m\in N;F_{\pord(\pred(p,1))}(3^{k(\pend(\pred(p,1)))}_{1+};K))$ and $\true(A_0(m);0)$ hold,}\\[40pt]
                  \pred(p,0) & \parbox[t]{8cm}{\raggedright if $\prule(p)=\exists_{m,A}$ and if $\false(m\in N;F_{\pord(\pred(p,0))}(3^{k(\pend(\pred(p,0)))}_{1+};K))$,}\\[20pt]
                  \pred(p,1) & \text{in all other cases}.
                 \end{cases}\\[7pt]
\stepbound(p,K):=F_{\pord(\stepproof(p,K))}(3^{k(\pend(\stepproof(p,K)))}_{1+};K).
\end{gather*}
To avoid confusion, let us stress that the third case of the case distinction only applies when a number $m$ with the stated property exists. Otherwise we are referred to the last case.\\
Building on these step-functions we define primitive recursive functions $\seqproof:\zet^\infty_0\times\mathbb N\times\mathbb N\rightarrow\zet^\infty_0$ and $\seqbound:\zet^\infty_0\times\mathbb N\times\mathbb N\rightarrow\mathbb N$ by the following simultaneous recursion:
\begin{align*}
 \seqproof(p,K,0) & :=p,\\
 \seqproof(p,K,n+1) & :=\stepproof(\seqproof(p,K,n),\seqbound(p,K,n)),\\[7pt]
 \seqbound(p,K,0) & :=K,\\
 \seqbound(p,K,n+1) & :=\stepbound(\seqproof(p,K,n),\seqbound(p,K,n)).
\end{align*}
\end{definition}

Before we can show that these functions have the desired properties, we need to exhibit the intended semantics of the step-down arguments that appear in the infinite proof system:

\begin{definition}\label{def:step-down-function}
 The primitive recursive function $\fund:\pord\times\mathbb N\times\mathbb N\rightarrow\pord$ is defined as follows
\begin{align*}
 \fund(\alpha,n,0) & :=\alpha,\\
 \fund(\alpha,n,x+1) & :=\{\fund(\alpha,n,x)\}(n).
\end{align*}
We use the abbreviations
\begin{equation*}
 \alpha\searrow_n^x\beta\quad:\Leftrightarrow\quad\exists_{y\leq x}\,\fund(\alpha,n,y)=\beta
\end{equation*}
and
\begin{equation*}
 \alpha\searrow_n\beta\quad:\Leftrightarrow\quad\exists_x\,\alpha\searrow_n^x\beta.
\end{equation*}
\end{definition}

The connection with step-down arguments is as follows:

\begin{proposition}[$\isigma_1$]\label{prop:step-down-from-argument}
If there is a step-down argument $s\vdash\beta\leq_n\alpha$ then we have $\alpha\searrow_n\beta$. In particular, Lemma \ref{lem:transformations-step-downs} assures that $\alpha\searrow_n 0$ holds for any $\alpha<\varepsilon_0$ and any number $n$.
\end{proposition}

This proposition will be proved in the next section. Here, we state how the relation $\alpha\searrow_n\beta$ relates to the fast-growing hierarchy. Recall that $F_\alpha(n)\!\downarrow$ abbreviates the formula $\exists_y\, F_\alpha(n)=y$.

\begin{lemma}\label{lem:connection-fast-step}
 The following holds:
\begin{enumerate}[label=(\alph*)]
 \item If we have $\alpha\searrow_n\beta$ and $F_\alpha(n)\!\downarrow$ then we have $F_\beta(n)\!\downarrow$ and $F_\beta(n)\leq F_\alpha(n)$.
 \item If we have $F_\alpha(n)\!\downarrow$ and $n'\leq n$ then we have $F_\alpha(n')\!\downarrow$ and $F_\alpha(n')\leq F_\alpha(n)$.
\end{enumerate}
\end{lemma}
\begin{proof}
 This is \cite[Lemma 2.3]{rathjen13}, building on \cite[Section 5.2]{sommer95}.
\end{proof}

We want to show that the function $(p,K)\mapsto(\stepproof(p,K),\stepbound(p,K))$ preserves bad pairs. Let us first single out a fact that will be used several times:

\begin{lemma}\label{lem:bad-pair-bound-inequality}
 Consider $K=F_\alpha(k)$ with $k\geq 2$, a step-down argument $s\vdash\beta+1\leq_k\alpha$, and a number $k'\leq F_\beta(k;K)$. Then we have
\begin{equation*}
 F_\beta(k';K)=F_\beta(k')<K.
\end{equation*}
In the presence of the other assumptions, the condition $k'\leq F_\beta(k;K)$ is weaker than the condition $k'\leq k$.
\end{lemma}
\begin{proof}
 Proposition \ref{prop:step-down-from-argument} gives $\alpha\searrow_k\beta+1$. By Lemma \ref{lem:connection-fast-step} the value $F_{\beta+1}(k)$ is defined and bounded by the number $K$. Using the $\Delta_0$-formula $F_\alpha^i(n)=m$ from \cite[Section 5.2]{sommer95} we can work with iterates of fast-growing functions. Thus \cite[Theorem 5.3]{sommer95} tells us that $F_\beta^{k+1}(k)$ is defined and equal to $F_{\beta+1}(k)$. Using \cite[Theorem 5.3, Proposition 5.4]{sommer95} and Lemma \ref{lem:connection-fast-step} we get the following chain of inequalities:
\begin{equation*}
 K\geq F_\beta^{k+1}(k)\geq F_\beta^3(k)>F_\beta^2(k)\geq F_\beta(k').
\end{equation*}
In particular all involved expressions are defined.\\
Concerning the alternative condition, the above shows that the conditions without $k'\leq F_\beta(k;K)$ imply that $F_\beta(k)$ is defined and bounded by $K$. Thus we have $F_\beta(k;K)=F_\beta(k)$, and \cite[Proposition 5.4]{sommer95} gives $k\leq F_\beta(k;K)$.
\end{proof}

Now we can show the promised result:

\begin{proposition}\label{lem:step-preserves-bad-pairs}
 If $(p,K)$ is a bad pair then $(\stepproof(p,K),\stepbound(p,K))$ is a bad pair as well. In this case we have $\stepbound(p,K)<K$.
\end{proposition}
\begin{proof}
The proof is by case distinction on $\prule(p)$. In all cases, Proposition \ref{prop:pred-proper-proof-proper} immediately yields $\stepproof(p,K)\in\zet^\infty$ and $\dcut(\stepproof(p,K))\leq\dcut(p)\leq 1$. The other verifications go as follows:\\
Let us first show that $\prule(p)=\ax$ is impossible: Indeed, by the local correctness conditions guaranteed by Proposition \ref{prop:pred-proper-proof-proper}, the assumption $\prule(p)=\ax$ would imply that $\pend(p)$ contains an axiom and that we have $\pord(p)\geq 2$. We show that this implies $\true(\pend(p);K)$, contradicting the fact that $(p,K)$ is a bad pair. If $\pend(p)$ contains a truth-axiom we immediately have $\true(\pend(p);K)$, independently of $K$. Now assume that $\pend(p)$ contains the $N$-axiom $\{n\notin N,(n+1)\in N\}$ (the axiom $\{n\notin N,n\in N\}$ is easier). We want to show $\true((n+1)\in N;K)$, or equivalently $3^{n+2}<K$. Since the formula $n\notin N$ occurs in $\pend(p)$ we have $n\leq k(\pend(p))$, and thus $3^{n+1}\leq 3_{1+}^{k(\pend(p))}$. Using \cite[Proposition 5.4]{sommer95} we get
\begin{equation*}
 3^{n+2}\leq(3^{n+1})^2\leq\left(3_{1+}^{k(\pend(p))}\right)^2<F_{\pord(p)}(3^{k(\pend(p))}_{1+})=K.
\end{equation*}
Next, observe that $\prule(p)$ cannot be of the form $\lor_{j,A}$, $\land_A$ or $\omega_A$: Otherwise, the local correctness condition due to Proposition \ref{prop:pred-proper-proof-proper} would force $\pend(p)$ to contain a disjunction, a conjunction or a universal formula, contradicting the assumption that $\pend(p)$ is a $\Sigma_1$-sequent (recall that in our setting all $\Delta_0$-formulas are atomic).\\
Let us continue with the case $\prule(p)=\acc^n$. There we have $\stepproof(p)=\pred(p,n)$, and thus $
 \pend(\stepproof(p,K))\subseteq\pend(p)$ by the local correctness of $p$. This ensures that $\pend(\stepproof(p,K))$ is a $\Sigma_1$-sequent. It also implies that we have $k(\pend(\stepproof(p,K)))\leq k(\pend(p))$ and thus 
\begin{equation*}
 2\leq 3^{k(\pend(\stepproof(p,K)))}_{1+}\leq 3^{k(\pend(p))}_{1+}.
\end{equation*}
The local correctness of $p$ (together with Lemma \ref{lem:step-down-properties}, in case $\dacc(p)=0$) also yields
\begin{equation*}
\step(p)\vdash\pord(\stepproof(p,K))+1\leq_{3^{k(\pend(p))}_{1+}}\pord(p).
\end{equation*}
In this situation Lemma \ref{lem:bad-pair-bound-inequality} gives
\begin{equation*}
 \stepbound(p,K)=F_{\pord(\stepproof(p,K))}(3^{k(\pend(\stepproof(p,K)))}_{1+})<K.
\end{equation*}
Finally, $\false(\pend(\stepproof(p,K));\stepbound(p,K))$ follows from the assumption $\false(\pend(p);K)$ by Lemma \ref{lem:truth-bound-monotonous}.\\
We come to the case $\prule(p)=\exists_{m,A}$. Since $A$ must occur in $\pend(p)$ this is only possible if $A$ is a $\Sigma_1$-formula. We distinguish two cases: First, assume that we have $\false(m\in N;F_{\pord(\pred(p,0))}(3^{k(\pend(\pred(p,0)))}_{1+};K))$. By definition we then have $\stepproof(p,K)=\pred(p,0)$. The local correctness of $p$ tells us
\begin{equation*}
\pend(\stepproof(p,K))\subseteq\pend(p)\cup\{m\in N\}.
\end{equation*}
This implies that $\pend(\stepproof(p,K))$ is a $\Sigma_1$-sequent and that we have
\begin{equation*}
 2\leq 3^{k(\pend(\stepproof(p,K)))}_{1+}\leq 3^{k(\pend(p))}_{1+}.
\end{equation*}
Also by the local correctness of $p$ we have
\begin{equation*}
\fund(0)*\step(p)\vdash\pord(\stepproof(p,K))+1\leq_{3^{k(\pend(p))}_{1+}}\pord(p),
\end{equation*}
and again Lemma \ref{lem:bad-pair-bound-inequality} yields
\begin{equation*}
\stepbound(p,K)=F_{\pord(\stepproof(p,K))}(3^{k(\pend(\stepproof(p,K)))}_{1+})<K.
\end{equation*}
To obtain the condition $\false(\pend(\stepproof(p,K));\stepbound(p,K))$ it suffices to show $\false(\pend(p);\stepbound(p,K))$ and $\false(m\in N;\stepbound(p,K))$. The first half follows from $\false(\pend(p);K)$ by Lemma \ref{lem:truth-bound-monotonous}, and the second half is the assumption of the case distinction.\\
Still for $\prule(p)=\exists_{m,A}$, assume now that we have
\begin{equation}\label{eq:case-distinction-existential-alternative}
 \true(m\in N;F_{\pord(\pred(p,0))}(3^{k(\pend(\pred(p,0)))}_{1+};K)).
\end{equation}
By definition we then have $\stepproof(p,K)=\pred(p,1)$, so that the local correctness of $p$ yields
\begin{equation*}
 \pend(\stepproof(p,K))\subseteq\pend(p)\cup\{A_0(m)\}.
\end{equation*}
This tells us that $\pend(\stepproof(p,K))$ is a $\Sigma_1$-sequent. Recall that the special formulas $n\in N$ and $n\notin N$ are not allowed as building blocks of compound formulas. In particular, the subformula $A_0(m)$ of $A$ is a proper $\Delta_0$-formula, and again we have
\begin{equation*}
 2\leq 3^{k(\pend(\stepproof(p,K)))}_{1+}\leq 3^{k(\pend(p))}_{1+}.
\end{equation*}
As above we can conclude
\begin{equation*}
\stepbound(p,K)=F_{\pord(\stepproof(p,K))}(3^{k(\pend(\stepproof(p,K)))}_{1+})<K.
\end{equation*}
It remains to establish $\false(\stepproof(p,K);\stepbound(p,K))$, which is easily reduced to $\false(A_0(m);K)$. The local correctness of $p$ implies that the formula $A\equiv\exists_x A_0(x)$ occurs in $\pend(p)$. Thus the assumption $\false(\pend(p);K)$ implies $\false(A;K)$. In view of Definition \ref{def:truth-finite-universe} this means that we cannot at the same time have $\true(m\in N;K)$ and $\true(A_0(m);K)$. So to obtain $\false(A_0(m);K)$ we only need to see $\true(m\in N;K)$. This follows from (\ref{eq:case-distinction-existential-alternative}) and the inequality $F_{\pord(\pred(p,0))}(3^{k(\pend(\pred(p,0)))}_{1+};K)\leq K$. Note that $F_\alpha(n;K)\leq K$ holds holds by the definition of the function $(\alpha,n,K)\mapsto F_\alpha(n;K)$.\\
It remains to treat the case $\prule(p)=\cut_A$. By the local correctness of $p$ we have $\dcut(\prule(p))\leq\dcut(p)\leq 1$. Thus $A$ must be in the class $\Delta_0\cup\Sigma_1$. Let us consider the different possibilities: The case where $A$ is a proper $\Delta_0$-formula is easy, because then $\true(A;K)$ is independent of $K$. Now assume that $A$ is of the form $n\in N$. We have to look at two cases: First, assume that we have $\true(n\in N;F_{\pord(\pred(p,1))}(3^{k(\pend(\pred(p,1)))}_{1+};K))$, and thus $\stepproof(p,K)=\pred(p,0)$. Let us begin with some preparations: The assumption of the case distinction is equivalent to
\begin{equation*}
 3^{n+1}<F_{\pord(\pred(p,1))}(3^{k(\pend(\pred(p,1)))}_{1+};K)).
\end{equation*}
Invoking Lemma \ref{lem:pred-ordinal-independent-n}, we may replace $\pord(\pred(p,1))$ by $\pord(\pred(p,0))$, i.e.\ by $\pord(\stepproof(p,K))$. Furthermore, the local correctness of $p$ gives
\begin{equation*}
 \pend(\pred(p,1))\subseteq\pend(p)\cup\{n\in N\},
\end{equation*}
which implies $3^{k(\pend(\pred(p,1)))}_{1+}\leq 3^{k(\pend(p))}_{1+}$. Using \cite[Proposition 5.4]{sommer95} it is easy to deduce that the inequality
\begin{equation}\label{eq:bounds-comp-cut-auxiliary}
 3^{n+1}<F_{\pord(\stepproof(p,K))}(3^{k(\pend(p))}_{1+};K))
\end{equation}
holds as well. After this preparation, let us come to the required verifications: The local correctness of $p$ gives
\begin{equation}\label{eq:inclusion-sequents-cut-auxiliary}
 \pend(\stepproof(p,K))=\pend(\pred(p,0))\subseteq\pend(p)\cup\{n\notin N\}.
\end{equation}
This means that $\pend(\stepproof(p,K))$ is a $\Sigma_1$-sequent, and it entails
\begin{equation}\label{eq:bounds-comp-cut-cases}
 3^{k(\pend(\stepproof(p,K)))}_{1+}\leq\max\{3^{k(\pend(p))}_{1+},3^{n+1}\}.
\end{equation}
Furthermore, the local correctness of $p$ yields
\begin{equation*}
\step(p)\vdash\pord(\stepproof(p,K))+1\leq_{3^{k(\pend(p))}_{1+}}\pord(p),
\end{equation*}
I claim that we have
\begin{equation*}
 3^{k(\pend(\stepproof(p,K)))}_{1+}\leq F_{\pord(\stepproof(p,K))}(3^{k(\pend(p))}_{1+};K)).
\end{equation*}
This is shown by a case distinction on the maximum in (\ref{eq:bounds-comp-cut-cases}). One case is provided by (\ref{eq:bounds-comp-cut-auxiliary}). The other case holds by the second part of Lemma \ref{lem:bad-pair-bound-inequality}. Having established this, the first part of Lemma \ref{lem:bad-pair-bound-inequality} provides the desired
\begin{equation*}
\stepbound(p,K)=F_{\pord(\stepproof(p,K))}(3^{k(\pend(\stepproof(p,K)))}_{1+})<K.
\end{equation*}
Finally, $\false(\pend(p);K)$ implies $\false(\pend(\stepproof(p,K));\stepbound(p,K))$ by Lemma \ref{lem:truth-bound-monotonous}. To see that this is the case, recall that the special formula $n\notin N$ is not in the class $\Delta_0$ (nor in $\Sigma_1$), so (\ref{eq:inclusion-sequents-cut-auxiliary}) does indeed imply
\begin{equation*}
 \pend(\stepproof(p,K))\cap(\Delta_0\cup\Sigma_1)\subseteq\pend(p).
\end{equation*}
Still concerning $\prule(p)=\cut_A$ with $A\equiv (n\in N)$, assume now that we have $\false(n\in N;F_{\pord(\pred(p,1))}(3^{k(\pend(\pred(p,1)))}_{1+};K))$. Then we have $\stepproof(p,K)=\pred(p,1)$ and thus
\begin{equation*}
 \pend(\stepproof(p,K))\subseteq\pend(p)\cup\{n\in N\}.
\end{equation*}
This tells us that $\pend(\stepproof(p,K))$ is a $\Sigma_1$-sequent and that we have
\begin{equation*}
 3^{k(\pend(\stepproof(p,K)))}_{1+}\leq 3^{k(\pend(p))}_{1+}.
\end{equation*}
As before we can use Lemma \ref{lem:bad-pair-bound-inequality} to get
\begin{equation*}
\stepbound(p,K)=F_{\pord(\stepproof(p,K))}(3^{k(\pend(\stepproof(p,K)))}_{1+})<K.
\end{equation*}
Finally, to establish $\false(\stepbound(p,K);\stepbound(p,K))$ it suffices to show $\false(\pend(p);\stepbound(p,K))$ and $\false(n\in N;\stepbound(p,K))$. The first conjunct reduces to the assumption $\false(\pend(p);K)$ by Lemma \ref{lem:truth-bound-monotonous}. The second conjunct is the assumption of the case distinction\\
It remains to treat the case $\prule(p)=\cut_A$ with $A\equiv\exists_x A_0(x)\in\Sigma_1$. Again, we have to distinguish two cases: Assume first that there is an $m\leq K$ such that we have
\begin{equation}\label{eq:sigma-cut-case-distinction}
 \true(m\in N;F_{\pord(\pred(p,1))}(3^{k(\pend(\pred(p,1)))}_{1+};K))\quad\text{and}\quad\true(A_0(m);0).
\end{equation}
Let $m$ be minimal with this property. By definition we then have $\stepproof(p,K)=\mathcal I_{m,\neg A}\pred(p,0)$. As a preparation, observe that the first conjunct of (\ref{eq:sigma-cut-case-distinction}) gives
\begin{equation*}
 3^{m+1}<F_{\pord(\pred(p,1))}(3^{k(\pend(\pred(p,1)))}_{1+};K).
\end{equation*}
Using Lemma \ref{lem:pred-ordinal-independent-n} we have
\begin{equation}\label{eq:inversion-leaves-ordinal}
 \pord(\pred(p,1))=\pord(\pred(p,0))=\pord(\stepproof(p,K)).
\end{equation}
Also, the local correctness of $p$ implies
\begin{equation*}
\pend(\pred(p,1))\subseteq\pend(p)\cup\{A\}
\end{equation*}
and thus
\begin{equation*}
 3^{k(\pend(\pred(p,1)))}_{1+}\leq 3^{k(\pend(p))}_{1+}.
\end{equation*}
We then obtain the inequality
\begin{equation}\label{eq:sigma-cut-bound-m}
 3^{m+1}<F_{\pord(\stepproof(p,K))}(3^{k(\pend(p))}_{1+};K),
\end{equation}
which we will need later. Coming to the required verifications, observe that we have
\begin{multline*}
 \pend(\stepproof(p,K))=\pend(\pred(p,0))\backslash\{\neg A\}\cup\{m\notin N,\neg A_0(m)\}\subseteq\\
\subseteq\pend(p)\cup\{m\notin N,\neg A_0(m)\}.
\end{multline*}
This shows that $\pend(\stepproof(p,K))$ is a $\Sigma_1$-sequent. Also, since the proper $\Delta_0$-formula $\neg A_0(m)$ cannot be of the form $n\notin N$, it tells us 
\begin{equation}\label{eq:sigma-inversion-case-distinction-maximum}
 3^{k(\pend(\stepproof(p,K)))}_{1+}\leq\max\{3^{k(\pend(p))}_{1+},3^{m+1}\}.
\end{equation}
Using (\ref{eq:inversion-leaves-ordinal}), the local correctness of $p$ implies
\begin{equation*}
 \step(p)\vdash\pord(\stepproof(p,K))+1\leq_{3^{k(\pend(p))}_{1+}}\pord(p).
\end{equation*}
I claim that we have
\begin{equation*}
 3^{k(\pend(\stepproof(p,K)))}_{1+}\leq F_{\pord(\stepproof(p,K))}(3^{k(\pend(p))}_{1+};K).
\end{equation*}
This is established by a case distinction on the maximum in (\ref{eq:sigma-inversion-case-distinction-maximum}): In one case it suffices to invoke (\ref{eq:sigma-cut-bound-m}). In the other case one uses the second half of Lemma \ref{lem:bad-pair-bound-inequality}. Having shown this, we can now invoke the first half of Lemma \ref{lem:bad-pair-bound-inequality} to get
\begin{equation*}
 \stepbound(p,K)=F_{\pord(\stepproof(p,K))}(3_{1+}^{k(\pend(\stepproof(p,K)))})<K.
\end{equation*}
It remains to establish $\false(\pend(\stepproof(p,K));\stepbound(p,K))$. Recall that the formula $m\notin N$ does not belong to the class $\Delta_0\cup\Sigma_1$. Thus it is enough to show $\false(\pend(p);\stepbound(p,K))$ and $\false(\neg A_0(m);\stepbound(p,K))$. The first conjunct follows from Lemma \ref{lem:truth-bound-monotonous}. To get the second conjunct it suffices to see that the proper $\Delta_0$-formula $A_0(m)$ is true, and this is the case according to (\ref{eq:sigma-cut-case-distinction}).\\
Still for $\prule(p)=\cut_A$ with $A\equiv\exists_x A_0(x)\in\Sigma_1$, assume now that (\ref{eq:sigma-cut-case-distinction}) holds for no $m\leq K$. Then we have $\stepproof(p,K)=\pred(p,1)$, and the local correctness of $p$ gives
\begin{equation*}
 \pend(\stepproof(p,K))\subseteq\pend(p)\cup\{A\}.
\end{equation*}
Thus $\pend(\stepproof(p,K))$ is indeed a $\Sigma_1$-sequent, and we have
\begin{equation*}
 3_{1+}^{k(\pend(\stepproof(p,K)))}\leq 3_{1+}^{k(\pend(p))}.
\end{equation*}
The local correctness of $p$ also yields
\begin{equation*}
 \step(p)\vdash\pord(\stepproof(p,K))+1\leq_{3_{1+}^{k(\pend(p))}}\pord(p).
\end{equation*}
By Lemma \ref{lem:bad-pair-bound-inequality} we get
\begin{equation*}
 \stepbound(p,K)=F_{\pord(\stepproof(p,K))}(3_{1+}^{k(\pend(\stepproof(p,K)))})<K.
\end{equation*}
Finally, the condition $\false(\pend(\stepproof(p,K));\stepbound(p,K))$ is easily reduced to $\false(\exists_x A_0(x);\stepbound(p,K))$. The latter is nothing but the assumption of the case distinction.
\end{proof}

For applications, we will use the following consequence of this work:

\begin{corollary}[$\isigma_1$]\label{cor:no-bad-pairs}
 There is no bad pair.
\end{corollary}
\begin{proof}
 Aiming at a contradiction, assume that $(p,K)$ is a bad pair. By induction on $n$ we can show that the following holds for all $n\in \mathbb N$:
\begin{itemize}
 \item $(\seqproof(p,K,n),\seqbound(p,K,n))$ is a bad pair
 \item if $n>0$ then $\seqbound(p,K,n)<\seqbound(p,K,n-1)$
\end{itemize}
The induction step is provided by Lemma \ref{lem:step-preserves-bad-pairs}. Having established this we can inductively prove
\begin{equation*}
 \seqbound(p,K,n)<(K+1)\dotminus n.
\end{equation*}
In particular this implies $\seqbound(p,K,K+1)<0$, which is absurd.
\end{proof}

We can deduce the main result of this paper, stating that $F_{\omega_n}\!\downarrow$ implies the uniform $\Sigma_1$-reflection principle over the theory $\isigma_n$:

\begin{theorem}[$\isigma_1$]\label{thm:totality-implies-reflection}
Assume that the function $F_{\omega_n}$ is total, with $n\geq 2$. Then any closed $\Pi_2$-formula that is provable in $\isigma_n$ is true.
\end{theorem}
\begin{proof}
As usual, it suffices to consider $\Sigma_1$-formulas: For a false provable $\Pi_2$-formula has a false instance, and this instance is still provable by specializing the universal quantifier. Under the assumption that $A$ is a false closed $\Sigma_1$-formula with $\isigma_n\vdash A$ we will construct a bad pair: Let $d^\Gamma\in\zet_n^0$ be the given finite proof with $\Gamma\subseteq\{A\}$. Recall that $\height(d)$ and $\dterm(d)$ denote the height and the term depth of $d$, respectively. By \cite[Proposition 5.4]{sommer95} the totality of $F_{\omega_n}$ implies that $F_\alpha$ is total for any ordinal smaller than $\omega_n$. We can thus define $p\in\zet^\infty_0$ and $K\in\mathbb N$ as
\begin{gather*}
 p :=\underbrace{\mathcal E\cdots\mathcal E}_{n-2\text{ times}}\mathcal E_0[d]_{\dterm(d)},\\
 K :=F_{\omega^{2\cdot\height(d)+1}_{n-1}}(3^{\max\{1,\dterm(d)\}+1}).
\end{gather*}
Let us prove $p\in\zet^\infty$: First, $[d]_{\dterm(d)}\in\zet^\infty$ is immediate. Also, we have
\begin{equation}\label{eq:ordinal-embedded-finite-proof}
 \pord([d]_{\dterm(d)})=\omega\cdot(2\cdot\height(d)+1)<\omega^2
\end{equation}
and $\dacc([d]_{\dterm(d)})=0$. It follows that $\mathcal E_0[d]_{\dterm(d)}$ is in $\zet^\infty$. Then $p\in\zet^\infty$ holds without any further conditions.\\
Next, observe that $d\in\zet_n^0$ implies $\dcut([d]_{\dterm(d)})=\dcut(d)\leq n$. Then we have $\dcut(\mathcal E_0[d]_{\dterm(d)})\leq n-1$ and finally $\dcut(p)\leq n-(n-1)=1$.\\
Concerning the end-sequent, we have
\begin{equation*}
 \pend(p)=\pend([d^\Gamma]_{\dterm(d)})=\Gamma\cup\{n\notin N\,|\, n\leq\dterm(d)\}\subseteq\{A\}\cup\{n\notin N\,|\, n\leq\dterm(d)\}.
\end{equation*}
Thus $\pend(p)$ is a $\Sigma_1$-sequent and we have $k(\pend(p))=\max\{1,\dterm(d)\}$.\\
Building on (\ref{eq:ordinal-embedded-finite-proof}) we have
\begin{equation*}
 \pord(\mathcal E_0[d]_{\dterm(d)})=3^{\omega\cdot(2\cdot\height(d)+1)}=\omega^{2\cdot\height(d)+1}
\end{equation*}
and then
\begin{equation*}
 \pord(p)=\omega_{n-1}^{2\cdot\height(d)+1}.
\end{equation*}
This implies $K=F_{\pord(p)}(3_{1+}^{k(\pend(p))})$, as required of a bad pair.\\
Finally, we need to show $\false(\pend(p);K)$. Using Lemma \ref{lem:truth-bound-monotonous} this can be reduced to $\false(A;K)$. The latter holds since we have assumed that $A$ is false outright.
\end{proof}

As a further application we deduce a variant of \cite[Lemma 3.5]{rathjen13}. Our proof is essentially that of \cite[Appendix A]{rathjen13}. Note that \cite[Lemma 3.5]{rathjen13} refers to the G\"odel numbers of proofs. Here, we instead use the height, the term-depth and the cut rank of a proof to measure its size.

\begin{theorem}[$\isigma_1$]
 Assume that $\feps(n)$ is defined, with $n\geq 1$. Then $F_{\omega_n}(n)$ is defined, and no $d\in\bigcup_n\zet_n^0$ can satisfy all of the following:
\begin{itemize}
 \item $d$ proves the empty sequent, i.e.\ a contradiction,
 \item $\dcut(d)\leq n$,
 \item $2\cdot\height(d)\leq F_{\omega_n}(n)$,
 \item $3^{\max\{1,\dterm(d)\}+1}\leq F_{\omega_n}(n)$.
\end{itemize}
\end{theorem}

Since the bounds on height and term-depth are so large, the result could easily be transferred to a different formal system (possibly containing function symbols for addition and multiplication) by a primitive recursive translation of proofs.

\begin{proof}
 Aiming at a contradiction, assume that $d$ is one of the proofs that the theorem aims to exclude. Based on this assumption, we construct a bad pair.\\
The main task is to show that $F_{\omega^{2\cdot\height(d)+1}_{n-1}}(3^{\max\{1,\dterm(d)\}+1})$ is defined. For this purpose we construct step-down arguments $s_n\vdash\omega_n+1\leq_1 \omega_{n+1}$ for all $n\in\mathbb N$ (cf.\ \cite[Lemma 2.13]{rathjen13}). In case $n=0$ we can take $s_0:=\fund(1)^\omega_2$. In the recursion step, $s_{n+1}$ arises as the following composition: First, one uses $\omega^{s_n}$ to descend from $\omega_{n+2}=\omega^{\omega{n+1}}$ to $\omega^{\omega_n+1}$. Using the step-down argument $\fund(1)$ we reach $\omega^{\omega_n}\cdot 2=\omega_{n+1}\cdot 2$. The step-down argument $\omega_{n+1}+s^{\omega_{n+1}}_1$ (see Lemma \ref{lem:transformations-step-downs}) then takes us to $\omega_{n+1}+1$. Now recall that, by definition, $\feps(n)$ is equal to $F_{\omega_{n+1}}(n)$. By Proposition \ref{prop:step-down-from-argument} we have $\omega_{n+1}\searrow_n\omega_n+1$. An application of Lemma \ref{lem:connection-fast-step} tells us that $F_{\omega_n+1}(n)$ is defined, and so is $F_{\omega_n}(n):=K_0$. From \cite[Theorem 5.3, Proposition 5.4]{sommer95} we know
\begin{equation*}
 F_{\omega_n}(K_0)\leq F_{\omega_n}^{n+1}(n)=F_{\omega_n+1}(n),
\end{equation*}
and all involved expressions are defined. In view of $\{\omega_n\}(K_0)=\omega_{n-1}^{K_0+1}$ we know that $F_{\omega_{n-1}^{K_0+1}}(K_0)=F_{\omega_n}(K_0)$ is defined. Since we have $2\cdot\height(d)\leq K_0$ it is not hard to give a step-down argument for $\omega^{2\cdot\height(d)+1}_{n-1}\leq_0 \omega_{n-1}^{K_0+1}$. Parallel to the above we can conclude that $F_{\omega^{2\cdot\height(d)+1}_{n-1}}(K_0)$ is defined. Finally, Lemma \ref{lem:connection-fast-step} tells us that $F_{\omega^{2\cdot\height(d)+1}_{n-1}}(3^{\max\{1,\dterm(d)\}+1})$ is defined as well. Now set
\begin{gather*}
 p :=\underbrace{\mathcal E\cdots\mathcal E}_{n-2\text{ times}}\mathcal E_0[d]_{\dterm(d)},\\
 K :=F_{\omega^{2\cdot\height(d)+1}_{n-1}}(3^{\max\{1,\dterm(d)\}+1}).
\end{gather*}
Much as in the proof of Theorem \ref{thm:totality-implies-reflection} one verifies that $(p,K)$ is a bad pair. This contradicts Corollary \ref{cor:no-bad-pairs}.
\end{proof}

Concerning a third application, recall that the computation sequence for $F_{\omega_n}(m)$ must terminate if $F_{\omega_n}(m)$ is to be defined. Thus the totality of $F_{\omega_n}$ implies that certain sequences of ordinals below $\omega_n$ cannot be infinitely descending. In a somewhat indirect manner we can now deduce that the same holds for any primitive recursive sequence of ordinals below $\omega_n$:

\begin{theorem}
 Let $f_p(x)=y$ be a $\Sigma_1$-formula in the variables $x,y,p$ such that we have
\begin{multline*}
 \isigma_1\vdash\text{``for all parameters $p$,}\\
\text{the formula $f_p(x)=y$ defines a function from $\mathbb N$ to the ordinals''}.
\end{multline*}
Then $\isigma_1$ proves the following: If $F_{\omega_n}$ with $n\geq 2$ is total then we have
\begin{equation*}
 f_p(0)<\omega_n\rightarrow\exists_m f_p(m+1)\geq f_p(m).
\end{equation*}
\end{theorem}
\begin{proof}
 In view of $\{\omega_n\}(l)=\omega_{n-1}^{l+1}$ it suffices to prove that the totality of $F_{\omega_n}$ implies
\begin{equation*}
 \forall_{k,l,p}(f_p(k)<\omega_{n-1}^{l+1}\rightarrow\exists_{m\geq k} f_p(m+1)\geq f_p(m)).
\end{equation*}
This follows from Theorem \ref{thm:totality-implies-reflection} if we can show
\begin{equation*}
 \forall_{n\geq 2}\forall_{l,p}\pr_{\isigma_n}(\forall_k(f_p(k)<\omega_{n-1}^{l+1}\rightarrow\exists_{m\geq k} f_p(m+1)\geq f_p(m))).
\end{equation*}
Indeed, by \cite[Lemma 4.3, Lemma 4.4]{sommer95} the theory $\isigma_n$ allows ordinal induction for $\Pi_2$-formulas up to $\omega_{n-1}^{l+1}$ (for any externally given $l$). It is easy to check that the formula
\begin{equation*}
 \forall_k(f_p(k)<\alpha\rightarrow\exists_{m\geq k} f_p(m+1)\geq f_p(m))
\end{equation*}
is progressive in $\alpha$.
\end{proof}

\section{Interpreting Step-Down Arguments in \texorpdfstring{$\isigma_1$}{ISigma-1}}

The goal of this section is to give a proof of Proposition \ref{prop:step-down-from-argument} in the theory $\isigma_1$. Many of our arguments come from \cite[Section 2]{rathjen13}. The frame theory there is Peano Arithmetic, but it is suggested that $\isigma_1$ is sufficient (see the passage after \cite[Lemma 2.4]{rathjen13}). To verify that this is the case we repeat the arguments in much detail. All ordinals that appear will be strictly below $\varepsilon_0$.\\
Recall the step-down function $\fund(\alpha,n,x)$ from Definition \ref{def:step-down-function}. We begin with an easy observation:

\begin{lemma}\label{lem:fund-sequences-monotone}
 For any ordinal $\alpha$ and any natural number $n$ we have
\begin{equation*}
 \fund(\alpha,n,x)>\fund(\alpha,n,y)\rightarrow x<y.
\end{equation*}
\end{lemma}
\begin{proof}
 It suffices to show the contrapositive of the statement, namely that
\begin{equation*}
 \fund(\alpha,n,y+z)\leq\fund(\alpha,n,y)
\end{equation*}
holds for all $z$. This is established by an easy induction over $z$. The induction step is due to the fact that fundamental sequences approximate ordinals from below.
\end{proof}

Next, we repeat \cite[Lemma 2.5]{rathjen13}:

\begin{lemma}\label{lem:trans-step-downs}
 If we have $\alpha\searrow_n^x\beta$ and $\beta\searrow_n^y\gamma$ then we have $\alpha\searrow_n^{x+y}\gamma$. If we have $\alpha\searrow_n^x \gamma$ and $\alpha\searrow_n\beta$ with $\beta\geq\gamma$ then we have $\beta\searrow_n^x \gamma$.
\end{lemma}
\begin{proof}
By induction on $y$ it is easy to show that we have
\begin{equation*}
 \fund(\alpha,n,x+y)=\fund(\fund(\alpha,n,x),n,y).
\end{equation*}
Coming to the first claim of the lemma, the assumptions $\alpha\searrow_n^x\beta$ and $\beta\searrow_n^y\gamma$ give us $x'\leq x$ and $y'\leq y$ with $\fund(\alpha,n,x')=\beta$ and $\gamma=\fund(\beta,n,y')$. By the above observation we get $\fund(\alpha,n,x'+y')=\gamma$, and the claim follows because of $x'+y'\leq x+y$. As for the second claim, the assumptions give us $x'\leq x$ and $x''$ with $\fund(\alpha,n,x')=\gamma$ and $\fund(\alpha,n,x'')=\beta$. The conclusion is trivial in case $\beta=\gamma$. Otherwise, the precedent lemma yields $x''<x'$ and we can deduce
\begin{multline*}
 \fund(\beta,n,x'-x'')=\fund(\fund(\alpha,n,x''),n,x'-x'')=\\
=\fund(\alpha,n,x''+(x'-x''))=\fund(\alpha,n,x')=\gamma,
\end{multline*}
just as required for $\beta\searrow_n^x \gamma$.
\end{proof}

We continue with a recapitulation of \cite[Lemma 2.7]{rathjen13}. Note that we include a converse to the result given there:

\begin{lemma}\label{lem:step-down-mesh}
 Assume that $\alpha$ meshes with $\beta$. For all ordinals $\gamma$ and natural numbers $n,x$ we have
\begin{equation*}
 \beta\searrow_n^x\gamma\qquad\Leftrightarrow\qquad\alpha+\beta\searrow_n^x\alpha+\gamma.
\end{equation*}
Furthermore, if we have $\alpha+\beta\searrow_n^x 0$ then we have $\alpha+\beta\searrow_n^x\alpha$. In that case $\beta\searrow_n^x 0$ follows by the first claim.
\end{lemma}
\begin{proof}
We first prove the following claim by induction on $y$:
\begin{multline*}
 (\fund(\beta,n,y)>0\,\lor\,\fund(\alpha+\beta,n,y)>\alpha)\\
\Rightarrow\quad\alpha+\fund(\beta,n,y+1)=\fund(\alpha+\beta,n,y+1)
\end{multline*}
In the case $y=0$, both the assumption $\beta=\fund(\beta,n,0)>0$ and the assumption $\alpha+\beta=\fund(\alpha+\beta,n,0)>\alpha$ imply $\beta>0$. From the definition of fundamental sequences (see e.g.\ \cite[Definition 1.2]{rathjen13}) one can see that we have
\begin{equation*}
 \alpha+\fund(\beta,n,1)=\alpha+\{\beta\}(n)=\{\alpha+\beta\}(n)=\fund(\alpha+\beta,n,1).
\end{equation*}
Coming to the induction step, observe that $\fund(\beta,n,y+1)>0$ entails
\begin{equation*}
 \fund(\beta,n,y)\geq\{\fund(\beta,n,y)\}(n)=\fund(\beta,n,y+1)>0,
\end{equation*}
and that $\fund(\alpha+\beta,n,y+1)>\alpha$ entails
\begin{equation*}
 \fund(\alpha+\beta,n,y)\geq\{\fund(\alpha+\beta,n,y)\}(n)=\fund(\alpha+\beta,n,y+1)>\alpha.
\end{equation*}
Thus, if $\fund(\beta,n,y+1)>0$ and $\fund(\alpha+\beta,n,y+1)>\alpha$ holds then the induction hypothesis yields $\alpha+\fund(\beta,n,y+1)=\fund(\alpha+\beta,n,y+1)$. In particular we have $\fund(\beta,n,y+1)>0$ in either case. Also, it is easy to check $\fund(\beta,n,y+1)\leq\beta$, which implies that $\alpha$ meshes with $\fund(\beta,n,y+1)$. Then, parallel to the base case, we have
\begin{multline*}
 \alpha+\fund(\beta,n,y+2)=\alpha+\{\fund(\beta,n,y+1)\}(n)=\{\alpha+\fund(\beta,n,y+1)\}(n)\stackrel{\text{IH}}{=}\\
\stackrel{\text{IH}}{=}\{\fund(\alpha+\beta,n,y+1)\}(n)=\fund(\alpha+\beta,n,y+2).
\end{multline*}
Coming to the first claim of the lemma, let us argue for the direction ``$\Leftarrow$'': Assume that we have $\alpha+\beta\searrow_n^x\alpha+\gamma$. Thus there is a $z\leq x$ with $\fund(\alpha+\beta,n,z)=\alpha+\gamma$. We distingish two cases: If we have $\alpha+\beta=\alpha+\gamma$ then we have $\beta=\gamma$, and $\beta\searrow_n^x\gamma$ follows immediately. If we have $\alpha+\beta\neq\alpha+\gamma$ then there must be some $y<z$ with $\fund(\alpha+\beta,n,y)\neq\fund(\alpha+\beta,n,y+1)=\alpha+\gamma$. In view of
\begin{equation*}
 \fund(\alpha+\beta,n,y)\geq\{\fund(\alpha+\beta,n,y)\}(n)=\fund(\alpha+\beta,n,y+1)
\end{equation*}
we can conclude
\begin{equation*}
 \fund(\alpha+\beta,n,y)>\fund(\alpha+\beta,n,y+1)=\alpha+\gamma\geq\alpha.
\end{equation*}
Thus the claim from the beginning of the proof tells us
\begin{equation*}
 \alpha+\fund(\beta,n,y+1)=\fund(\alpha+\beta,n,y+1)=\alpha+\gamma.
\end{equation*}
This implies $\fund(\beta,n,y+1)=\gamma$, and in view of $y<z\leq x$ we see $\beta\searrow_n^x\gamma$. The direction ``$\Rightarrow$'' is similar and slightly easier.\\
The remaining claim is shown by contraposition: Assuming that $\alpha+\beta\searrow_n^x\alpha$ fails, we show that $\fund(\alpha+\beta,n,z)>\alpha$ holds for all $z\leq x$. For $z=0$ it suffices to note that we must have $\alpha+\beta\neq\alpha$ and thus $\alpha+\beta>\alpha$. In the induction step, the induction hypothesis provides the assumption to the claim at the top of this proof. We then get
\begin{equation*}
 \fund(\alpha+\beta,n,z+1)=\alpha+\fund(\beta,n,z+1)\geq\alpha.
\end{equation*}
The assumption that $\alpha+\beta\searrow_n^x\alpha$ fails implies $\fund(\alpha+\beta,n,z+1)\neq\alpha$, so
\begin{equation*}
 \fund(\alpha+\beta,n,z+1)>\alpha
\end{equation*}
must indeed hold.
\end{proof}

Next, we review \cite[Lemma 2.10, Lemma 2.11]{rathjen13}, showing how step-downs can be lifted to powers of $\omega$:

\begin{lemma}\label{lem:step-down-lift-omega}
 Given $\omega^\alpha\searrow_n^x 0$ and $\alpha\searrow_n\beta$ we can conclude $\omega^\alpha\searrow_n^x \omega^\beta$ and $\alpha\searrow_n^x 0$.
\end{lemma}
\begin{proof}
 Write $e(\gamma)$ for the exponent of the leading summand of the Cantor normal form of $\gamma$. We agree on $e(0)=0$. Assuming that we have $\alpha\searrow_n\beta$ we prove the following claim by induction over $y$: We have
\begin{equation*}
 \alpha\searrow_n^y e(\fund(\omega^\alpha,n,y))\qquad\text{and}\qquad\Big(\,\omega^\alpha\searrow_n^y \omega^\beta\quad\text{or}\quad\fund(\omega^\alpha,n,y)>\omega^\beta\,\Big).
\end{equation*}
Concerning $y=0$, we first observe $e(\fund(\omega^\alpha,n,0))=e(\omega^\alpha)=\alpha$, and indeed $\alpha\searrow_n^0 \alpha$ holds. As for the second part of the claim, if $\omega^\alpha\searrow_x^0 \omega^\beta$ fails the we must have $\omega^\alpha\neq\omega^\beta$. Furthermore, $\alpha\searrow_n\beta$ allows us to conclude $\alpha\geq\beta$. Together we obtain $\fund(\omega^\alpha,n,0)=\omega^\alpha>\omega^\beta$.\\
We come to the step $y\leadsto y+1$. In case $\fund(\omega^\alpha,n,y)=0$ we also have $\fund(\omega^\alpha,n,y+1)=0$ and the induction step is easily deduced. Otherwise we can write $\fund(\omega^\alpha,n,y)=\omega^\gamma+\delta$ where $\omega^\gamma$ and $\delta$ mesh. Concerning the first half of the claim, the induction hypothesis yields $\alpha\searrow_n^y\gamma$. In case $\delta>0$ we have
\begin{equation*}
 \fund(\omega^\alpha,n,y+1)=\{\fund(\omega^\alpha,n,y)\}(n)=\{\omega^\gamma+\delta\}(n)=\omega^\gamma+\{\delta\}(n),
\end{equation*}
and thus $e(\fund(\omega^\alpha,n,y+1))=\gamma$ remains unchanged. Clearly, we still have $\alpha\searrow_n^{y+1}\gamma$. In case $\delta=0$ we observe that
\begin{equation*}
 e(\fund(\omega^\alpha,n,y+1))=e(\{\omega^\gamma\}(n))=\{\gamma\}(n)
\end{equation*}
holds, no matter if $\gamma$ is zero, a successor or a limit ordinal. Thus what we need is $\alpha\searrow_n^{y+1}\{\gamma\}(n)$, and this follows easily from $\alpha\searrow_n^y\gamma$.\\
Still concerning the induction step, let us come to the second half of the claim. As above we write $\fund(\omega^\alpha,n,y)=\omega^\gamma+\delta$. If $\omega^\alpha\searrow_n^{y+1} \omega^\beta$ holds then we are done. Otherwise $\omega^\alpha\searrow_n^y \omega^\beta$ must fail as well, so the induction hypothesis implies $\gamma\geq\beta$, and indeed $\gamma>\beta$ in case $\delta=0$. Still under the assumption that $\omega^\alpha\searrow_n^{y+1} \omega^\beta$ fails, the goal $\fund(\omega^\alpha,n,y+1)>\omega^\beta$ reduces to
\begin{equation*}
\fund(\omega^\alpha,n,y+1)\geq\omega^\beta.
\end{equation*}
For $\delta>0$ this holds by
\begin{equation*}
 \fund(\omega^\alpha,n,y+1)=\{\omega^\gamma+\delta\}(n)=\omega^\gamma+\{\delta\}(n)\geq\omega^\gamma\geq\omega^\beta.
\end{equation*}
For $\delta=0$, observe that $\gamma>\beta$ excludes the case $\gamma=0$. In the case of a successor or limit ordinal we have
\begin{equation*}
 \fund(\omega^\alpha,n,y+1)=\{\omega^\gamma\}(n)\geq\omega^{\{\gamma\}(n)},
\end{equation*}
so it remains to establish $\{\gamma\}(n)\geq\beta$: The first part of the induction hypothesis gives $\alpha\searrow_n^y\gamma$. Also, we have the assumptions $\alpha\searrow_n\beta$ and $\gamma>\beta$. By Lemma \ref{lem:trans-step-downs} we obtain $\gamma\searrow_n\beta$. Now $\gamma\neq\beta$ implies $\{\gamma\}(n)\searrow_n\beta$, and thus indeed $\{\gamma\}(n)\geq\beta$.\\
To deduce the claim of the lemma, assume that we have $\omega^\alpha\searrow_x^n 0$ and $\alpha\searrow_n\beta$. The first assumption implies $\fund(\omega^\alpha,n,x)=0$, which makes $\fund(\omega^\alpha,n,x)>\omega^\beta$ impossible. Thus we must indeed have $\omega^\alpha\searrow_n^x \omega^\beta$ and $\alpha\searrow_n^y e(0)=0$.
\end{proof}

To proceed, we need the auxiliary result of \cite[Lemma 2.8]{rathjen13}:

\begin{lemma}\label{lem:step-down-finite-multiples}
 If we have $\omega^\alpha\cdot k\searrow_n^x 0$ and $l\leq k$ then we have $\omega^\alpha\cdot k\searrow_n^x \omega^\alpha\cdot l$.
\end{lemma}
\begin{proof}
 We prove the contrapositive of the lemma: If we do not have $\omega^\alpha\cdot k\searrow_n^x \omega^\alpha\cdot l$ then $\fund(\omega^\alpha\cdot k,n,z)>\omega^\alpha\cdot l$ holds for all $z\leq x$. This is shown by induction on $z$: For $z=0$ it suffices to observe that we must have $\omega^\alpha\cdot k\neq \omega^\alpha\cdot l$ and thus $\omega^\alpha\cdot k>\omega^\alpha\cdot l$. In the induction step the induction hypothesis implies that we have
\begin{equation*}
 \omega^\alpha\cdot l<\fund(\omega^\alpha\cdot k,n,z)\leq\omega^\alpha\cdot k<\omega^{\alpha+1}.
\end{equation*}
Thus $\fund(\omega^\alpha\cdot k,n,z)$ is of the form
\begin{equation*}1
 \fund(\omega^\alpha\cdot k,n,z)=\omega^\alpha\cdot l+\beta
\end{equation*}
where $\omega^\alpha\cdot l$ meshes with $\beta>0$. As in the proof of Lemma \ref{lem:step-down-mesh} we can conclude
\begin{equation*}
\fund(\omega^\alpha\cdot k,n,z+1)=\{\fund(\omega^\alpha\cdot k,n,z)\}(n)=\omega^\alpha\cdot l+\{\beta\}(n)\geq\omega^\alpha\cdot l.
\end{equation*}
The assumption that $\omega^\alpha\cdot k\searrow_n^x \omega^\alpha\cdot l$ fails tells us that $\fund(\omega^\alpha\cdot k,n,z+1)=\omega^\alpha\cdot l$ is impossible, so that we indeed get the required strict equality.
\end{proof}

We come to an important result which is very similar to \cite[Lemma 2.12]{rathjen13}:

\begin{lemma}\label{lem:increase-setp-down}
 Given $\alpha\searrow_n^x 0$ and $m\leq n$ we have $\alpha\searrow_n^x\{\alpha\}(m)$.
\end{lemma}
\begin{proof}
 The proof is by induction over the height of the Cantor normal form of the ordinal $\alpha$, with all other parameters fixed. If $\alpha$ is zero or a successor ordinal we have $\{\alpha\}(m)=\{\alpha\}(n)$ and the claim is immediate (note that for $x=0$ the assumption $\alpha\searrow_n^x 0$ forces $\alpha=0$). If $\alpha$ is a limit ordinal we can write $\alpha=\beta+\omega^\gamma$ where $\beta$ meshes with $\omega^\gamma$ and we have $\gamma>0$ (or we have $\alpha=\omega^\gamma$, which is easier). We distinguish two cases: First, assume that $\gamma=\delta+1$ is a successor ordinal. Then we have $\{\alpha\}(n)=\beta+\omega^\delta\cdot(n+1)$, and $\alpha\searrow_n^x 0$ implies $\beta+\omega^\delta\cdot(n+1)\searrow_n^{x-1} 0$. Since $\beta$ still meshes with $\omega^\delta\cdot(n+1)<\omega^\gamma$ Lemma \ref{lem:step-down-mesh} gives $\omega^\delta\cdot(n+1)\searrow_n^{x-1} 0$. Now Lemma \ref{lem:step-down-finite-multiples} tells us $\omega^\delta\cdot(n+1)\searrow_n^{x-1} \omega^\delta\cdot(m+1)$. By the other direction of Lemma \ref{lem:step-down-mesh} we get $\beta+\omega^\delta\cdot(n+1)\searrow_n^{x-1} \beta+\omega^\delta\cdot(m+1)$, and finally $\alpha\searrow_n^x \beta+\omega^\delta\cdot(m+1)=\{\alpha\}(m)$.\\
 We come to the other case, where $\gamma$ is a limit ordinal. From $\alpha\searrow_n^x 0$ and Lemma \ref{lem:step-down-mesh} we learn $\omega^\gamma\searrow_n^x 0$. From Lemma \ref{lem:step-down-lift-omega} we learn $\gamma\searrow_n^x 0$. Since the height of the Cantor normal form of $\gamma$ is smaller than the height of the Cantor normal form of $\alpha$ the induction hypothesis yields $\gamma\searrow_n^x\{\gamma\}(m)$. Another application of Lemma \ref{lem:step-down-lift-omega} gives $\omega^\gamma\searrow_n^x \omega^{\{\gamma\}(m)}$. Finally, we can invoke Lemma \ref{lem:step-down-mesh} to get $\alpha\searrow_n^x \beta+\omega^{\{\gamma\}(m)}=\{\alpha\}(m)$.
\end{proof}

Finally, we need to know that step-downs to a constant member of the fundamental sequence always end in zero:

\begin{lemma}\label{lem:all-step-downs-to-zero}
 For any $\alpha<\varepsilon_0$ and any $n$ we have $\alpha\searrow_n 0$.
\end{lemma}
\begin{proof}
 Our proof is inspired by \cite[Proposition 2.9]{solovay81}. First, we define a primitive recursive function which maps a pair of an ordinal and a natural number to a natural number. The idea is to replace any occurrence of $\omega$ in the (hereditary) Cantor normal form of $\alpha$ by the number $n+2$. Formally, the definition is by structural induction over the Cantor normal form of ordinals: We set $\text{Num}(0,n):=0$ and
\begin{equation*}
 \text{Num}(\omega^{\alpha_1}\cdot m_1+\dots +\omega^{\alpha_k}\cdot m_k,n):=(n+2)^{\text{Num}(\alpha_1,n)}\cdot m_1+\dots +(n+2)^{\text{Num}(\alpha_k,n)}\cdot m_k.
\end{equation*}
Observe that we have $\text{Num}(\gamma+1,n)=\text{Num}(\gamma,n)+1$. By structural induction over the Cantor normal form of $\alpha$ one can then show that we have
\begin{equation*}
 \text{Num}(\alpha,n)>\text{Num}(\{\alpha\}(n),n)\qquad\text{whenever $\alpha>0$}.
\end{equation*}
Aiming at a contradiction, assume that $\alpha\searrow_n^{\text{Num}(\alpha,n)} 0$ fails. This means that we have $\fund(\alpha,n,x)>0$ and thus
\begin{equation*}
 \text{Num}(\fund(\alpha,n,x),n)>\text{Num}(\fund(\alpha,n,x+1),n)\qquad\text{for all $x\leq\text{Num}(\alpha,n)$}.
\end{equation*}
In view of $\text{Num}(\fund(\alpha,n,0),n)=\text{Num}(\alpha,n)$ we can use induction to arrive at
\begin{equation*}
 \text{Num}(\fund(\alpha,n,\text{Num}(\alpha,n)+1),n)<0,
\end{equation*}
which is absurd.
\end{proof}

Putting the pieces together we can show the open result:

\begin{proof}[Proof of Proposition \ref{prop:step-down-from-argument}]
In view of Lemma \ref{lem:all-step-downs-to-zero} it suffices to prove the following claim:
\begin{equation*}
 \parbox{11cm}{Consider a step-down argument $s$ and a number $n\geq\stepba(s)$. If we have $\stepto(s)\searrow_n^x 0$ then we have $\stepto(s)\searrow_n^x \stepbo(s)$.}
\end{equation*}
The proof is by structural induction over step-down arguments, which are finite terms. The parameters $n$ and $x$ are fixed throughout. We start with the base case $s=\fund(m)^\alpha_{\{\alpha\}(m)}$, with $\stepto(s)=\alpha$ and $\stepbo(s)=\{\alpha\}(m)$. The assumption $n\geq\stepba(s)$ amounts to $n\geq m$. Then the claim is nothing but Lemma \ref{lem:increase-setp-down}.\\
Next, consider a step-down argument of the form $s=s^\gamma_{\stepbo(s)}*s^{\stepto(s)}_\gamma$. Note that we have $\stepba(s)=\max\{\stepba(s^\gamma_{\stepbo(s)}),\stepba(s^{\stepto(s)}_\gamma)\}$. Thus the induction hypothesis gives $\stepto(s)\searrow_n^x\gamma$. Having established this much, Lemma \ref{lem:trans-step-downs} yields $\gamma\searrow_n^x 0$, and another application of the induction hypothesis provides $\gamma\searrow_n^x\stepbo(s)$. Now Lemma \ref{lem:trans-step-downs} gives $\stepto(s)\searrow_n^{2x}\stepbo(s)$. In view of the assumption $\stepto(s)\searrow_n^x 0$ Lemma \ref{lem:fund-sequences-monotone} improves this to $\stepto(s)\searrow_n^{x}\stepbo(s)$.\\
We come to a step-down argument $s=\alpha+s^\beta_\alpha$ where $\alpha$ and $\beta$ mesh. Note that we have $\stepto(s)=\alpha+\beta$ and $\stepto(s)=\alpha+\gamma$. Given $\alpha+\beta\searrow_n^x 0$ Lemma \ref{lem:step-down-mesh} implies $\beta\searrow_n^x 0$, so that the induction hypothesis yields $\beta\searrow_n^x\gamma$. Again by Lemma \ref{lem:step-down-mesh} we ge the desired $\alpha+\beta\searrow_n^x\alpha+\gamma$.\\
Finally, consider an argument $s=\omega^{s^\alpha_\beta}$ with $\stepto(s)=\omega^\alpha$ and $\stepbo(s)=\omega^\beta$. Given $\omega^\alpha\searrow_n^x 0$ Lemma \ref{lem:step-down-lift-omega} yields $\alpha\searrow_n^x 0$. The induction hypothesis allows to conclude $\alpha\searrow_n^x \beta$. Another application of Lemma \ref{lem:step-down-lift-omega} gives the desired $\omega^\alpha\searrow_n^x\omega^\beta$.
\end{proof}

\bibliographystyle{alpha}
\bibliography{Reflection-Fragments-Freund}

\end{document}